\documentclass[11pt,reqno]{amsart} 
\usepackage{amsmath,amssymb,amsthm}
\usepackage{mathscinet} 
\usepackage[usenames,dvipsnames]{xcolor}
\usepackage[mathscr]{eucal}

\usepackage[latin1]{inputenc}

\usepackage{syntonly}  

\usepackage{dsfont} 
\usepackage{graphicx}
\usepackage{xcolor}
\usepackage{enumitem}

\newtheorem{theorem}{Theorem}

\newtheorem{proposition}[theorem]{Proposition}
\newtheorem{definition}[theorem]{Definition}
\newtheorem{corollary}[theorem]{Corollary}
\newtheorem{lemma}[theorem]{Lemma}

\newtheorem{remark}{Remark}
\newtheorem{example}[remark]{Example}

\numberwithin{equation}{section}
\def\qed{\hbox{${\vcenter{\vbox{		 
   \hrule height 0.4pt\hbox{\vrule width 0.4pt height 6pt
   \kern5pt\vrule width 0.4pt}\hrule height 0.4pt}}}$}}

\newcommand{\commentout}[1]{}

\newcommand{\comment}[1]{$\null$}

\begin{document}

\title[\emph{Stochastic vorticity equation in $\mathbb R^2$ with not regular noise}]{\emph{Stochastic vorticity equation in $\mathbb R^2$ with not regular noise}}

\author{Benedetta Ferrario, Margherita Zanella}
\address{Benedetta Ferrario, Margherita Zanella \newline Universit\`a di Pavia, Dipartimento di Matematica ''F. Casorati'', via Ferrata 5, 27100 Pavia, Italy}
\email{benedetta.ferrario@unipv.it, margherita.zanella01@ateneopv.it}


\begin{abstract}
We consider the Navier-Stokes equations in vorticity form in $\mathbb{R}^2$ with a white noise forcing term 
of multiplicative type, whose spatial covariance  is not regular enough to apply the It\^o calculus in $L^q$ spaces, 
$1<q<\infty$. We prove the existence of a unique strong (in the probability sense) solution.
\end{abstract}

\maketitle

\section{Introduction}

The aim of this work is to study the stochastic equation for the vorticity, 
 that is the equation describing  the local rotation of a viscous incompressible fluid with a random forcing term. 
Formally, the equation for the vorticity is obtained by taking the curl of the stochastic Navier-Stokes equations, that is the equations of motion of a viscous incompressible fluid with a random forcing term. These latter equations are given by
\begin{equation}
\label{NS_0}
\begin{cases}
\displaystyle  \partial_t v+[- \nu \Delta v+(v \cdot \nabla )v+ \nabla p]\ dt 
       = G(v)\, \partial_tW
\\ 
\nabla \cdot  v=0,
\end{cases}
\end{equation}
where the unknowns are the vector velocity $v=v(t,x)$ and the scalar pressure $p=p(t,x)$. 
By $\nu>0$ we denote the viscosity coefficient; for simplicity, from now on we assume $\nu=1$. 
In our model the stochastic force depends on the velocity itself. We consider $x \in \mathbb{R}^2$ and $t \in \left[0,T\right]$, for a fixed $T>0$. The above equations are associated with the initial condition 
\begin{equation*}
v(0,x)=v_0(x).
\end{equation*}
\newline
We set $\xi = \nabla^{\perp} \cdot v$, where the curl operator is given by 
$ \nabla^{\perp}=\left(-\frac{\partial\;}{\partial x_2},\frac{\partial\;}{\partial x_1}\right)$. 
The scalar unknown $\xi=\xi(t,x)$ represents the vorticity of the fluid; 
it satisfies the following equations
\begin{equation}
\label{vort}
\begin{cases}
\displaystyle  \partial_t \xi+[- \Delta \xi+v \cdot \nabla \xi]\ dt
       = \nabla^\perp \cdot ( G(v)\, \partial_t W)
       \\
       \nabla \cdot v=0
\\ 
\xi= \nabla^{\perp}\cdot v, 
\end{cases}
\end{equation}
associated with the initial condition
$$
\xi(0,x)=\xi_0(x)
$$
where $\xi_0=\nabla^\perp\cdot v_0$.

In equation \eqref{vort} it appears the velocity $v$. This can be expressed in terms of 
the vorticity $\xi$ by means of the Biot-Savart law $v=k *\xi$. On $\mathbb{R}^2$ 
the Biot-Savart kernel $k$ is given by (see \cite[Chapter 2.1]{Majda2002})
\begin{equation}
\label{BS_intr}
k(x)=- \frac{1}{2\pi}\frac{x^{\perp}}{|x|^2}, 
\end{equation}
with the natural notation $x^{\perp}=(-x_2, x_1)$. In this way, \eqref{vort} can be written as a closed equation for the vorticity as we did for instance in  \cite{FZ}.
On a non compact domain this closed form is difficult to handle; indeed, $k \notin 
L^p(\mathbb R^2)$ for any $1\le p\le \infty$. 
Then, we take into account the equations \eqref{NS_0} for the velocity. 
When $v$ exists and has a suitable regularity, we can handle  the equations \eqref{vort} for the vorticity.

The problem of the existence and uniqueness of ${L}^2$-solutions of the stochastic Navier-Stokes equations \eqref{NS_0} has been addressed by many authors. There is also a consistent literature on more regular solutions, but the majority of the work is limited to bounded domains 
(see e.g. \cite{AlbFla, K} and the therein references). An extension to unbounded domains is not trivial since the direct application of the compactness method, which is central in the proof, fails. Main source of difficulty is the fact that the embedding of the Sobolev space of functions with square integrable gradient into the ${L}^2$-space, unlike in the bounded space, is not compact. 
To address this problem different ideas have been employed. One way is to introduce weighted Sobolev spaces, as done for instance in \cite{VisFur, CapPsz}. This method allows in particular to consider spatially homogeneous noises. 
A different approach is used in \cite{MikRoz2004}, \cite{MikRoz} and \cite{BrzMot}. 
In \cite{MikRoz} the authors prove the existence of an ${L}^2$-valued continuous 
solution considering a more general noise than in \cite{CapPsz}. 
Their proof is based on some compactness and tightness criteria in local spaces and in the space 
${L}^2$ with the weak topology.  Differently, in \cite{MikRoz2004} also the vorticity is considered, 
but the results involve $v$ and $\xi$ in $L^p(\mathbb R^d)$ for $p>d$.
Inspired by \cite{MikRoz}, in \cite{BrzMot} the authors prove existence and uniqueness of a strong 
$L^2$-solution, by means of
a modification of the classical Dubinsky compactness theorem that allows to work in unbounded domains. 

Following the same approach of \cite{BrzMot}, in \cite{BrzFer} authors impose really weak assumptions on the covariance operator of the noise term in the velocity equation \eqref{NS_0}. In particular, it is not regular enough to allow to use It\^o formula in the space of finite energy velocity vectors, which is the basic space in which one looks for existence of solutions.

Inspired by \cite{BrzFer} we consider the vorticity equations \eqref{vort} with a multiplicative 
noise whose covariance is not regular enough to allow to use the It\^o formula in $L^q$ spaces, for $1<q<\infty$; in particular, the covariance of the noise is not a trace class operator in the space of 
finite energy vorticity  and this case has not been considered in  previous papers.
The aim of this work is to prove the existence of a martingale solution for the vorticity equation \eqref{vort} in $\mathbb{R}^2$ when $v_0,\xi_0\in L^2(\mathbb R^2)$, which is not considered by \cite{MikRoz2004}. We ask minimal assumptions on the covariance of the noise. Moreover, we prove pathwise uniqueness; this implies existence of a strong solution too. 
A more regular solution will be found when $v_0,\xi_0\in L^2(\mathbb R^2)\cap L^q(\mathbb R^2)$ for $q>2$.
The results are proved by working directly on the equation for the vorticity \eqref{vort}
and using suitable estimates on $v$ coming from equations \eqref{NS_0}.

As far as the contents of the paper are concerned, in Section \ref{mat_set} we define the abstract setting in order to write \eqref{NS_0} and \eqref{vort} as It\^o equations in some Banach space. 
In Section \ref{v_sec} we are concerned with the study of the regularity of the velocity solution to equations \eqref{NS_0}. 
In Section \ref{vor_sec_bad} we prove the existence and uniqueness of a strong solution to the vorticity equations \eqref{vort}. 
An existence result for equations \eqref{NS_0} driven by a more regular random forcing term is given in Appendix \ref{A}.

 \begin{remark}
 As we are working in the intersection of analysis and probability, the terminology concerning the notion of solution can cause some confusion. 
When we talk about strong and weak solutions we understand them in a probabilistic sense. In the case of strong solutions, the underlying probability space is given in advance. On the other side, in the case of a martingale solution the stochastic basis is constructed as part of the solution. In both cases solutions are weak in the sense of PDEs since we test them against smooth functions.
\end{remark}
\smallskip

{\small Notation. In the sequel, we shall indicate with $C$ a constant that may varies from line to line. In certain cases, we write $C_{\alpha,\beta,\dots}$ to 
emphasize the dependence of the constant on the parameters $\alpha, \beta, \dots$.}

\section{Mathematical framework}
\label{mat_set}
\subsection{Functional spaces}
We first introduce the functional spaces.

Let $q \in \left[1, \infty\right)$ and $d=1,2$. 
Let $L^q=\left[L^q(\mathbb{R}^2) \right]^d$ with
norm 
\begin{equation*}
\|v\|_{L^q}=\left(\sum_{k=1}^d \int_{\mathbb{R}^2}\left|v_k(x)\right|^q\, {\rm d}x \right)^{\frac 1q}
\end{equation*}
where $v=(v_1,...v_d)$. Similarly,  $L^{\infty}=\left[L^{\infty}(\mathbb{R}^2) \right]^d$ is the  Banach 
space  with norm
\[
\|v\|_{L^{\infty}}
=\sum_{k=1}^d \text{ess sup} \{|v_k(x)|, x \in \mathbb{R}^2\}.
\]
If $q=2$, then $L^2$ is a Hilbert space with scalar product given by 
\begin{equation*}
\langle u, v \rangle_{L^2}=\sum_{k=1}^d\int_{\mathbb{R}^2}u_k(x) v_k(x)\, {\rm d}x.
\end{equation*}
By $\mathbb{L}^q$, $1 \le q \le \infty$ we denote the spaces
\begin{equation}
\label{MatbbLq}
\mathbb{L}^q=\{u \in L^q : \nabla \cdot u=0\}
\end{equation}
with norm inherit from $L^q$. The divergence has to be understood in the weak sense.
Notice that we use the same notation $L^q$ for scalar fields ($d=1)$ and vector fields ($d=2$). 
The context shall make clear the case we are considering. We will specify the dimension $d$ only in some ambiguous cases.

For $s \in \mathbb{R}$ and $1 \le q \le \infty$, set $J^s=(I-\Delta)^{\frac s2}$. We define the generalized Sobolev spaces as
\begin{equation}
\label{151017}
W^{s,q}=\{u \in\mathcal{S}'(\mathbb{R}^d) : \|J^su\|_{L^q} < \infty\}
\end{equation}
and the generalized Sobolev spaces of divergence free vector distributions as
\begin{equation}
\label{151017bis}
H^{s,q}=\{u \in \left[W^{s,q}\right]^d : \nabla \cdot u=0\}.
\end{equation}
We have (see \cite{BerLof}) that $J^{\sigma}$ is an isomorphism between $W^{s,q}$ and $W^{s-\sigma, q}$. 
For $s_1< s_2$ there is the continuous embedding $W^{s_2,q} \subset 
W^{s_1,q}$ and the dual space of $W^{s,q}$ is $W^{-s,p}$ with $1 < p \le \infty$, $\frac 1p+\frac 1q =1$. We denote by $\langle \cdot, \cdot \rangle$ the $W^{s,q}-W^{-s,p}$ duality bracket:
\begin{equation*}
\langle u,v \rangle =\sum_{k=1}^d \int_{\mathbb{R}^2} (J^su_k)(x)(J^{-s}v_k)(x)\, {\rm d}x.
\end{equation*}
Let us focus on the case $s=1$, $q \in(1, \infty)$. The space $W^{1,q}$ is endowed with the norm
\begin{equation*}
\|u\|^q_{W^{1,q}}= \|u\|^q_{L^q}+ \|\nabla u\|^q_{L^q}.
\end{equation*}
Since we are on the whole space $\mathbb{R}^2$, Poincar\'e inequality does not hold; 
thus there is no  equivalence of the norms $\|u \|_{W^{1,q}}$ and $\|\nabla u\|_{L^q}$. Nevertheless, we have the following result (see \cite[Lemma 3.1]{BrzPsz}).
\begin{lemma}
\label{lemma_curl}
Let $q \in (1, \infty)$. There is a constant $C$ such that $\|\nabla v\|_{L^q} \le C \| \emph{curl} \ v\|_{L^q}$ for every $v \in W^{1,q}$.
\end{lemma}
In particular, since $\|\text{curl} v \|_{L^q} \le  \|\nabla v\|_{L^q}$ we get the equivalence of the norms 
\begin{equation}
\label{equivnorm}
\|\nabla u\|_{L^q}\sim\|\text{curl}\ u\|_{L^q}.
\end{equation}
In the sequel, when we ask $v_0\in \mathbb L^q$, $\xi_0\in L^q$ this is equivalent to $v_0 \in H^{1,q}$.

We recall some Sobolev embedding theorems (see \cite[Theorem 9.12 and Corollary 9.11]{brezis2010functional}). For every $q\in (2, \infty)$ the space $W^{1,q}$ is continuously embedded into $L^{\infty}$, namely there exists a constant $C$ (depending on $q$) such that:
\begin{align}
\label{Sobolev}
\| v\|_{L^{\infty}} 
\le  C\| v\|_{W^{1,q}}.
\end{align}
For every $q\in\left[2, \infty\right)$, $W^{1,2}$ is continuously embedded in $L^q$, 
namely there exists a constant $C$ (depending on $q$) such that:
\begin{align*}\label{Sobolev3}\| v\|_{L^q} \le  C\| v\|_{W^{1,2}}.\end{align*}

Since $\mathbb{R}^2$ is an unbounded domain, the embedding of $H^{1,2}$ into $\mathbb{L}^2$ is not compact. However, by \cite[Lemma 2.5]{HolWic} (see also \cite[Lemma C.1]{BrzMot}), there exists a separable Hilbert space $\mathbb{U}$ such that
\begin{equation*}
\mathbb{U} \subset H^{1,2} \subset \mathbb{L}^2,
\end{equation*}
the embedding $i$ of $\mathbb{U}$ into $H^{1,2}$ being dense and compact. Then we have 
\begin{equation}
\label{mathbbU}
\mathbb{U} \underset{i}{\subset} H^{1,2} \subset \mathbb{L}^2 \simeq (\mathbb{L}^2)^* \subset H^{-1,2} \underset{i^*}{\subset} \mathbb{U}^*
\end{equation}
where $(\mathbb{L}^2)^*$ and $H^{-1,2}$ are the dual spaces of $\mathbb{L}^2$ and $H^{1,2}$ respectively, $(\mathbb{L}^2)^*$ being identified with $\mathbb{L}^2$ and $i^*$ is the dual operator to the embedding $i$. Moreover, $i^*$ is compact as well.
\newline
The same considerations hold also when we consider the spaces $W^{1,2}$ and $L^2$. In this case we shall denote by $U$ the Hilbert space such that $U \subset W^{1,2} \subset L^2$.

By $C^{\infty}_{sol}:=\left[C^{\infty}_{sol}(\mathbb{R}^2)\right]^2$ we denote the space consisting of all divergence free vectors $v \in \left[C^{\infty}(\mathbb{R}^2)\right]^2$ with compact support.
We denote by $L^2(0,T;\mathbb{L}^2_{loc})$ the space of measurable functions $v : \left[0,T\right]\rightarrow \mathbb{L}^2$ such that, for any $R>0$, the norm 
$\|u\|_ {L^2(0,T;\mathbb{L}^2_R)}=\left(\int_0^T\int_{|x|<R}|u(t,x)|^2\, {\rm d}x\, {\rm d}t\right)^{\frac12}$ is finite. 
It is a Fr\'echet space with the topology generated by the seminorms $\|u\|_ {L^2(0,T;\mathbb{L}^2_R)}$, $R \in \mathbb{N}$.
\newline
We denote by $C(\left[0,T\right];L^2_w)$ the space of $L^2$-valued weakly continuous functions with the topology of uniform weak convergence on $\left[0,T\right]$; in particular $v_n \rightarrow v$ in $C(\left[0,T\right];L^2_w)$ means
\begin{equation}
\lim_{n \rightarrow \infty} \sup_{0 \le t \le T} |\langle v_n(t)-v(t),h\rangle_{L^2}|=0
\end{equation}
for all $h \in L^2$.
For $q \ge 2$, we denote by $L^{\infty}_w(0,T;L^q)$ the space $L^{\infty}(0,T;L^q)$ with the weak-$*$ topology.
\newline
For $0<\beta<1$ by $C^{\beta}(\left[0,T\right];H^{s,2})$ we denote the Banach space of  $H^{s,2}$-valued $\beta$-H\"older continuous functions endowed with the following norm 
\begin{equation*}
\|u\|_{C^{\beta}(\left[0,T\right];H^{s,2})} 
= \sup_{0 \le t \le T} \|u(t)\|_{H^{s,2}}+ \sup_{0 \le s < t \le T}\frac{\|u(t)-u(s)\|_{H^{s,2}}}{|t-s|^{\beta}}.
\end{equation*}

\subsection{Operators}

We define the operators that will appear in the abstract formulation of \eqref{NS_0} and \eqref{vort}.
We refer to \cite{Temam2001} and \cite{KatoPonce} for the details.

Let $A=-\Delta$. It is a linear unbounded operator in $W^{s,p}$ and $H^{s,p}$ ($s \in \mathbb{R}, 1 \le p < \infty$); it generates a contractive and analytic $C_0$-semigroup $\{S(t)\}_{t \ge 0}$. We have $A:H^{1,2} \rightarrow H^{-1,2}$ and 
\begin{equation*}
\langle Au,u\rangle= \|\nabla u\|^2_{L^2}, \qquad u \in H^{1,2}.
\end{equation*}

We define the bilinear vector operator $B:H^{1,2} \times H^{1,2} \rightarrow H^{-1,2}$
 as 
\begin{equation*}
\langle B(u,v), z\rangle = \int_{\mathbb{R}^2}(u(x) \cdot \nabla ) v(x)\cdot z(x)\, {\rm d}x.
\end{equation*}
The following lemma gathers the main properties of $B$ we shall need in the following.
\begin{lemma}
\label{lemmaB}
\begin{enumerate}[label=\roman{*}), ref=(\roman{*})]
\item The vector operator $B$ is bounded from $H^{1,2} \times H^{1,2}$ into $H^{-1,2}$.
\item It holds
\begin{equation}
\label{B0}
\langle B(u,v),z\rangle=-\langle B(u,z),v\rangle, \qquad \forall \ u,v,z \in H^{1,2}
\end{equation}
\begin{equation}
\label{B00}
\langle B(u,v),v\rangle=0, \qquad \forall \ u,v \in H^{1,2}
\end{equation}
\item For every $q>2$ it holds
\begin{equation}
\label{Bq}
\langle B(u,u),|u|^{q-2}u\rangle=0, \qquad  \forall \ u \in H^{1,2}.
\end{equation}
\item
$B$ can be extended to be a bounded operator from $\mathbb L^4\times \mathbb L^4$ to $H^{-1,2}$.
\end{enumerate}
\end{lemma}
\begin{proof}
i) follows by H\"older and Sobolev inequalities: we get 
\begin{equation}
\label{Bnew}
|\langle B(u,v),z\rangle| 
\le \|u\|_{\mathbb{L}^4}\|\nabla v\|_{\mathbb{L}^2}\|z\|_{\mathbb{L}^4}
\le C \|u\|_{H^{1,2}}\|v\|_{H^{1,2}}\|z\|_{H^{1,2}}.
\end{equation}
\eqref{B0} is obtained by the integration by parts formula; when $v=z$ we get \eqref{B00}.\\
Also iii) is obtained by the integration by parts formula; notice that the duality is well defined since 
\[
|\langle B(u,u),|u|^{q-2}u\rangle|\le
\|u\|_{L^4}\|\nabla u\|_{L^2}\||u|^{q-2}u\|_{L^4}
\]
and $H^{1,2}\subset \mathbb L^r$ for any finite $r$.\\
iv) comes from \eqref{B0} and the first estimates in \eqref{Bnew} by using  the fact that $H^{1,2}$ is dense in $L^4$.
\end{proof}

We define the bilinear scalar operator $F:H^{1,2}\times W^{1,2} \rightarrow W^{-1,2}$ as 
\begin{equation}
\langle F(u,\xi),\zeta\rangle=\int_{\mathbb{R}^2}u(x) \cdot \nabla \xi(x) \zeta(x)\, {\rm d}x.
\end{equation}
The following lemma gathers the main properties of $F$ we shall need in the following.
\begin{lemma}
\begin{enumerate}[label=\roman{*}), ref=(\roman{*})]
\item The operator $F$ is bounded from $H^{1,2}\times W^{1,2}$ into $W^{-1,2}$.
\item It holds
\begin{equation}
\label{F2}
\langle F(u,\xi),\zeta\rangle=-\langle F(u,\zeta),\xi\rangle, \qquad \langle F(u,\xi),\xi\rangle=0
\qquad \forall u \in H^{1,2}, \zeta,\xi \in W^{1,2}.
\end{equation}
\item 
For every $q>2$ we get
\begin{equation}
\label{Fq} 
\langle F(u, \xi), \zeta |\zeta|^{q-2} \rangle=-(q-1)\langle F(u,\zeta),|\zeta|^{q-2} \xi \rangle, \qquad  \forall \xi \in W^{1,2}, \
\zeta  \in L^{2(q-1)},  
u\in \mathbb{L}^{\infty} 
\end{equation}

and 
\begin{equation}
\label{Fq0}
\langle F(u, \xi), \xi |\xi|^{q-2} \rangle=0, \qquad  \forall \xi \in W^{1,2}, \ u \in \mathbb{L}^{\infty}
\end{equation}
\item $F$ can be extended to a bounded bilinear 
operator from $\mathbb{L}^4 \times L^4$ to $W^{-1,2}$ and
\begin{equation}
\label{F3}
\|F(u,\xi)\|_{W^{-1,2}} \le \|u\|_{\mathbb{L}^4}\|\xi\|_{L^4}.
\end{equation} 
\end{enumerate}
 \end{lemma}
 \begin{proof}
The proof of statements i), ii) and iv) can be done as in Lemma \ref{lemmaB}. Statement iii) is obtained by integrating by parts, where in \eqref{Fq} the proof is done first with smooth functions and then by density is extended on the spaces specified. Notice that the l.h.s. side of \eqref{Fq} is well defined since 
\begin{equation*}
|\langle F(u, \xi), \zeta |\zeta|^{q-2} \rangle|  \le \|u\|_{L^{\infty}}\|\nabla \xi\|_{L^2} \|\zeta\|_{L^{2(q-1)}}
\le \|u\|_{L^{\infty}} \|\xi\|_{W^{1,2}} \|\zeta\|_{L^{2(q-1)}}.
\end{equation*}
Eventually, \eqref{Fq0} is a particular case of \eqref{Fq}.
\end{proof}

\subsection{Random forcing term}
\label{RFT_bad}
We define the noise forcing term driving equation \eqref{NS_0}. Given a real separable Hilbert space $\mathcal{H}$, we consider a $\mathcal{H}$-cylindrical Wiener process $W$ defined on a stochastic basis $(\Omega, \mathcal{F}, \{\mathcal{F}_t\}_{t \in \left[0,T\right]},\mathbb{P})$, where $\{\mathcal{F}_t\}_{t \in \left[0,T\right]}$ is a 
complete right continuous filtration. We can write
\begin{equation}
\label{W}
W(t)=\sum_{k=1}^{\infty}\beta_k(t) h_k, \qquad t \in \left[0,T\right],
\end{equation}
where $\{\beta_k\}_{k \in \mathbb{N}}$ is a sequence of standard independent identically distributed Wiener processes defined on $(\Omega, \mathcal{F}, \{\mathcal{F}_t\}_{t \in \left[0,T\right]},\mathbb{P})$ and $\{h_k\}_{k\in \mathbb{N}}$ is a complete orthonormal system in $\mathcal{H}$.

We recall  basic facts concerning stochastic integration in Banach spaces. 
For more details see e.g. \cite{NeiPhD}, \cite{Det1991} and \cite{Neerven2007a}.

Let $E$ be a real separable Banach space.
We denote by $\gamma$ the standard Gaussian cylindrical distribution on $\mathcal{H}$.
A bounded linear operator $K\in \mathcal{L}(\mathcal H,E)$ is called $\gamma$-radonifying when the image $K(\gamma):=\gamma \circ K^{-1}$ of $\gamma$ under $K$ is $\sigma$-additive on the algebra of cylindrical sets in $E$. We set
\begin{equation*}
R(\mathcal H,E):=\{K \in \mathcal{L}(\mathcal H,E) \ \text{and $K$ is $\gamma$-radonifying}\}.
\end{equation*}
The algebra of cylindrical sets in $E$ generates the Borel $\sigma$-algebra, $\mathcal{B}(E)$ on $E$ (see \cite{Kuo}). Thus $K(\gamma)$ extends to a Borel measure on $\mathcal{B}(E)$ which we denote by $\gamma_K$. In particular, $\gamma_K$ is a Gaussian measure on $\mathcal{B}(E)$.
For $K \in R(\mathcal H,E)$ we put 
\begin{equation}
\label{normR}
\|K\|^2_{R(\mathcal H,E)}:=\int_E\|x\|^2_E\,{\rm d}\gamma_K(x). 
\end{equation}
As $\gamma_K$ is Gaussian, then by the Fernique-Landau-Shepp Theorem (see \cite{Kuo}), $\|K\|_{R(\mathcal H,E)}$ is finite. Moreover, see (see \cite{NeiPhD}), $R(\mathcal H,E)$ is a separable Banach space endowed with the norm \eqref{normR}.

If $E$ is a Hilbert space, then $K:\mathcal H\rightarrow E$ is $\gamma$-radonifying means that
 $K$ is Hilbert-Schmidt and the adjoint operator $T^*:E \rightarrow \mathcal H$ is Hilbert-Schmidt too.
We denote by $L_{HS}(\mathcal H;E)$ the space of all Hilbert-Schmidt operators from $\mathcal{H}$ into the (Hilbert) space $E$. In this case it holds $\|K\|_{L_{HS}(\mathcal H;E)}=\|K\|_{R(\mathcal H;E)}=\|K^*\|_{L_{HS}(E;\mathcal H)}$. 

We have the following characterization of $\gamma$-radonifying operators when $E=L^q$, see \cite[Proposition 13.7]{Van} and \cite[Theorem 2.3]{BrzNer}. 
\begin{proposition}
\label{pro4}
Let $1\le q<\infty$ and $\{h_j\}_{k=1}^{\infty}$ a complete orthonormal system in $\mathcal{H}$. For an operator $K\in \mathcal{L}(\mathcal H;L^q)$ the following two conditions are equivalent:
\begin{itemize}
\item $K \in R(\mathcal H,L^q)$;
\item $\left(\sum_{k=1}^{\infty}|Kh_k|^2\right)^{\frac 12} \in L^q$. 
\end{itemize}
Moreover, the norms $\|K\|_{R(\mathcal H;L^q)}$ and 
$\|\left(\sum_{k=1}^{\infty}|Kh_k|^2\right)^{\frac 12} \|_{L^q}$ are equivalent.
\end{proposition}

Let us fix $T>0$ and let $Y$ be a Banach space. Let us denote by $\mathcal{M}^p_{\mathcal{W}}(0,T;Y)$ the Banach space of all $\{\mathcal{F}_t\}_t$-predictable $Y$-valued processes $\Phi$ such that 
\begin{equation*}
\|\Phi\|_{\mathcal{M}^p_{\mathcal{W}}(0,T;Y)}:=\left( \mathbb{E}\int_0^T\|\Phi(t)\|^p_Y\, {\rm d}t \right)^{\frac1p}
\end{equation*}
is finite.
Given a process $\Phi$ in $\mathcal{M}^2_{\mathcal{W}}(0,T;R(\mathcal H,L^q))$, 
the stochastic integral 
$$X(t)=\int_0^t \Phi(s)\, {\rm d}\mathcal{W}(s)$$
is well defined (see \cite{NeiPhD} and \cite{Det1991} for more details and a theory of stochastic integration in a more general class of Banach spaces) and a Burkholder-Davis-Gundy type inequality holds (see \cite[Theorem 2.4, Theorem 3.3]{Det1991}).
\begin{theorem}
Let $1\le q<\infty$ and let $\mathcal{W}$ be a $\mathcal H$-cylindrical Wiener process.
If, for some $1\le m <\infty$ we have 
\begin{equation*}
\mathbb{E} \left[ \left(\int_0^T\|\Phi(t)\|^2_{R(\mathcal{H};L^q)}\, {\rm d}t\right)^{\frac m2}\right]<\infty
\end{equation*}
then $X$ has a progressively measurable $L^q$-valued version and there exists a positive constant $C_m$ such that 
\begin{equation}
\label{BDG}
\mathbb{E} \sup_{0\le t \le T}\|X(t)\|^m_{L^q} \le C_m \mathbb{E} \left[\left(\int_0^T \|\Phi(s)\|^2_{R(\mathcal{H},L^q)}\, {\rm d}s \right)^{\frac m2}\right].
\end{equation}  
\end{theorem}

On the covariance operator $G$ appearing in equation \eqref{NS_0} 
we make the following assumptions. We consider $q>2$ and assume that  there exists $g \in (0,1)$ such that 
\begin{description}
\item [(IG1)] 
The mapping $ G:\mathbb{L}^2 \rightarrow L_{\text{HS}}(\mathcal{H};H^{1-g,2})$ is well defined and 
\begin{equation*}
\sup_{v \in \mathbb{L}^2}\| G(v)\|_{L_{\text{HS}}(\mathcal{H};H^{1-g,2})}=:C_{g,2} < \infty,
\end{equation*}
\item[(IG2)] 
The mapping $G:\mathbb{L}^2 \rightarrow R(\mathcal{H};H^{1-g,q})$ is well defined and 
\begin{equation*}
\sup_{v \in \mathbb{L}^2}\|G(v)\|_{R(\mathcal{H};H^{1-g,q})}=: C_{g,q} < \infty.
\end{equation*}

\item[(IG3)] 
If assumption \textbf{(IG1)} holds, then
for any $\varphi \in H^{1-g,2}$ and any $v \in \mathbb{L}^2$ the mapping $v \rightarrow G(v)^*\varphi \in \mathcal{H}$ is continuous when in $\mathbb{L}^2$ we consider the Fr\'echet topology 
 inherited from the space $\mathbb{L}^2_{loc}$ or the weak topology of $\mathbb{L}^2$.
\item[(IG4)] For all $z\in C^{\infty}_{sol}$ 
the real valued function $v \mapsto \|G(v)^*z\|_{\mathcal{H}}$ is continuous on $H^{1,2}$ endowed with the strong $L^2$-topology.
\item[(IG5)] 
If assumption \textbf{(IG1)} holds, then $G$ ia s  Lipschitz continuous function  when we consider a weak norm,  i.e.
\begin{equation*}
\text{there exists} \  L_g>0:\|G(v_1)-G(v_2)\|_{L_{HS}(\mathcal{H};\mathbb{L}^2)} \le L_g \|v_1-v_2\|_{\mathbb{L}^2}
\end{equation*}
for any $v_1, v_2 \in \mathbb L^2$.

\end{description}
\begin{remark}
\label{rem1}
\begin{enumerate}[label=\roman{*}., ref=(\roman{*})]
\item A map $G:\mathbb{L}^2\rightarrow R(\mathcal{H};H^{1-g,q})$ is well defined iff the map $J^{1-g} G:\mathbb{L}^2\rightarrow R(\mathcal{H};\mathbb L^{q})$ is well defined. Moreover
\begin{equation*}
\|J^{1-g} G(v)\|_{R(\mathcal{H};\mathbb L^{q})} = \| G(v)\|_{R(\mathcal{H};H^{1-g,q})} < C_{g,q}, \qquad v \in \mathbb{L}^2.
\end{equation*}
\item From \eqref{BDG} and \textbf{(IG1)}, for any finite 
$m\ge 1$ we have
\begin{equation*}
\mathbb{E} \left\Vert \int_0^t  G(v(s))\, {\rm d}W(s) \right\Vert^m_{H^{1-g,2} }\le C_m(C_{g,2})^m t^{\frac m2}.
\end{equation*} 
\item If $G(v) \in L_{HS}(\mathcal{H};H^{1-g,2})$ with the uniform bound of \textbf{(IG1)}, then the same holds for the adjoint operator, i.e.
\begin{equation*}
\sup_{v \in \mathbb{L}^2}\|G(v)^*\|_{L_{HS}(H^{1-g,2};\mathcal{H})}=C_{g,2}.
\end{equation*}
\end{enumerate}
\end{remark}

The noise driving equation \eqref{vort} is obtained by taking the curl of the noise driving equation \eqref{NS_0}. Bearing in mind \eqref{W}, it is given by
\begin{equation}
\text{curl}(G(v)W(t))=\sum_{k=1}^{\infty}\beta_k(t) \text{curl}(G(v)h_k), \qquad t \in \left[0,T\right].
\end{equation}

Let $q>2$. Notice that, for all $v \in \mathbb{L}^2$ and $k \in \mathbb{N}$, $G(v)h_k \in H^{1-g,2}\cap H^{1-g,q}$. 
By taking the curl of this latter quantity we loose one order of differentiability, namely curl$(G(v)h_k) \in W^{-g,2} \cap W^{-g,q}$. Formally, we introduce the operator $\tilde G$ in the following way: given $v \in \mathbb{L}^2$, for all $\psi \in \mathcal{H}$, $\tilde G(v)(\psi):=\text{curl}(G(v)\psi)$. Thus we have that the mapping $\tilde G$ is well defined from $\mathbb{L}^2$ to $L_{\text{HS}}(\mathcal{H};W^{-g,2}) \cap R(\mathcal{H};W^{-g,q})$. 

Let us  notice  that an analogue of Remark \ref{rem1} holds.
\begin{remark}
\label{rem4}
We have that
\begin{enumerate}[label=\roman{*}., ref=(\roman{*})]
\item 
a map $\tilde G:\mathbb{L}^2\rightarrow R(\mathcal{H};W^{-g,q})$ is well defined iff the map $J^{-g} \tilde G:\mathbb{L}^2\rightarrow R(\mathcal{H};L^{q})$ is well defined. Moreover
\begin{equation*}
\|J^{-g} \tilde G(v)\|_{R(\mathcal{H}; L^{q})} = \| \tilde G(v)\|_{R(\mathcal{H};W^{-g,q})} < C_{g,q}, \qquad v \in \mathbb{L}^2.
\end{equation*}
\item
From \eqref{BDG} and \textbf{(IG1)}, for any finite $m\ge 1$ we have
\begin{equation}
\label{wpstoc}
\mathbb{E} \left\Vert \int_0^t  \tilde G(v(s))\, {\rm d}W(s) \right\Vert^m_{W^{-g,2} }\le C_m(C_{g,2})^m t^{\frac m2}.
\end{equation} 
\end{enumerate}
\end{remark}
Therefore, the assumptions on $G$ are transferred to $\tilde G$. For instance, we have
\begin{align*}
\|\tilde G(v)\|^2_{L_{HS}(\mathcal H;W^{-g,2})}
&= \sum_{k=1}^\infty \|\tilde G(v)h_k\|^2_{W^{-g,2}}
= \sum_{k=1}^\infty \|\text{curl } (G(v)h_k)\|^2_{W^{-g,2}}
\\
&\le \sum_{k=1}^\infty \|G(v)h_k\|^2_{H^{1-g,2}} 
=\|G(v)\|^2_{L_{HS}(\mathcal H;H^{1-g,2})},
\end{align*}
and, thanks to \eqref{equivnorm}, Proposition \ref{pro4} and Remark \ref{rem1}(i)
\begin{align*}
\|\tilde G(v)\|^q_{R(\mathcal H;W^{-g,q})}
&= \left(\mathbb{E} \left\Vert\sum_{k}\beta_k \tilde G(v)h_k\right\Vert^2_{W^{-g,q}} \right)^{\frac q2}
= \left(\mathbb{E} \left \Vert \tilde G(v)\sum_{k}\beta_k h_k\right\Vert^2_{W^{-g,q}} \right)^{\frac q2}
\\
&=\left(\mathbb{E} \left \Vert \text{curl}\left( G(v)\sum_{k}\beta_k h_k\right)\right\Vert^2_{W^{-g,q}} \right)^{\frac q2} 
\le \left(\mathbb{E} \left \Vert G(v)\sum_{k}\beta_k h_k\right\Vert^2_{H^{1-g,q}} \right)^{\frac q2}
\\
&= \left(\mathbb{E} \left \Vert \sum_{k}\beta_kG(v) h_k\right\Vert^2_{H^{1-g,q}} \right)^{\frac q2}
= \|G(v)\|^q_{R(\mathcal H;H^{1-g,q})}
\end{align*}

With a little abuse of notation we shall write 
$\tilde G(v){\rm d}W(t)$ instead of curl$(G(v){\rm d}W(t))$, where 
$\tilde G:= \text{curl}G$.
\newline
Let us notice that the set of assumptions made on the covariance operator $G$ are rather good to deal with equation \eqref{NS_0} in the spaces $\mathbb L^2$ or $\mathbb L^q$.
On the other hand, when we deal with the equation for the vorticity, 
we are concerned with a covariance operator not regular enough to use the It\^o calculus.

For the sake of clarity, among the above assumptions made on $G$, we rewrite in terms of $\tilde G$ those assumptions that we will use in the following. Let $0<g<1$ and $q>2$. Then
\begin{description}
\item[(I$\tilde G$1)] The mapping $\tilde G:\mathbb{L}^2 \rightarrow L_{HS}(\mathcal{H};W^{-g,2})$ is well defined and 
\begin{equation*}
\sup_{v \in \mathbb{L}^2}\|\tilde G(v)\|_{L_{HS}(\mathcal{H};W^{-g,2})}=:C_{g,2} < \infty.
\end{equation*}
\item[(I$\tilde G$2)] The mapping $\tilde G:\mathbb{L}^2 \rightarrow R(\mathcal{H};W^{-g,q})$ is well defined and 
\begin{equation*}
\sup_{v \in \mathbb{L}^2}\|\tilde G(v)\|_{R(\mathcal{H};W^{-g,q})}=:C_{g,q} < \infty.
\end{equation*}
\item[(I$\tilde G$3)] 
If assumption \textbf{(I$\tilde G$1)} holds, then for any $\varphi \in W^{-g,2}$ and any $v \in \mathbb{L}^2$ the mapping $v \rightarrow \tilde G(v)^*\varphi \in \mathcal{H}$ is continuous when in $\mathbb{L}^2$ we consider the Fr\'echet topology inherited from the space $\mathbb{L}^2_{loc}$ or the weak topology of $\mathbb{L}^2$.
\end{description}
\begin{example}
Let $G(v)h_k=c_k\sigma(v)e_k$ with $\{e_k\}_k$ a complete orthonormal system in $H^{1-g,2}$, $c_k \in\mathbb{R}$ and $\sigma:\mathbb{L}^2 \rightarrow \mathbb{R}$ such that 
\begin{align*}
&\sup_{v \in \mathbb{L}^2} |\sigma(v)|:=C_{\sigma}^1 < \infty,
\\
&\exists \ L>0: |\sigma(v_1)-\sigma(v_2)| \le L \|v_1-v_2\|_{\mathbb{L}^2}, \;
\forall v_1, v_2 \in\mathbb{L}^2 
\\
&\sigma(v_1)\rightarrow\sigma(v_2) \ \emph{if $v_1$ converges to $v_2$ in $H^{1,2}$ endowed with the strong $\mathbb{L}^2$ topology},
\\
&\sigma(v_1)\rightarrow\sigma(v_2) \ \emph{if $v_1$ converges to $v_2$ in $\mathbb{L}^2_w$ or $\mathbb{L}^2_{loc}$}.
\end{align*}

For instance, the above conditions on $\sigma$ are fulfilled for $\sigma(v)=\frac{\langle v,h\rangle^2}{1+\langle v,h\rangle^2}$ with a given $h \in \mathbb{L}^2$.
\newline
Condition \textbf{(IG1)} holds if and only if
\begin{equation}
\sum_{k=1}^{\infty}c_k^2< \infty
\end{equation}
and \textbf{(IG2)} hold if $e_k \in H^{1-g,q}$ and 
\begin{equation}
\sum_{k=1}^{\infty}c_k^2\|e_k\|^2_{H^{1-g,q}}< \infty.
\end{equation}
In order to prove \textbf{(IG3)} notice that $G(v)^*e_k= \sigma(v)c_kh_k$ for any $k$; therefore, given $\varphi \in H^{1-g,2}$ (with $\varphi =\sum_{k=1}^{\infty}\langle \varphi, e_k\rangle_{H^{1-g,2}}e_k$ and $\|\varphi\|^2_{H^{1-g,2}} = \sum_{k=1}^{\infty}|\langle \varphi, e_k\rangle_{H^{1-g,2}}|^2$)
\begin{align*}
\|G(v_1)^*\varphi-G(v_2)^*\varphi\|^2_{\mathcal{H}}
=&\left\Vert\sum_{k=1}^{\infty}\left[G(v_1)^*\langle \varphi, e_k\rangle_{H^{1-g,2}}e_k-G(v_2)^*\langle \varphi, e_k\rangle_{H^{1-g,2}}e_k \right] \right\Vert^2_{\mathcal{H}}
\\
&=\left(\sigma(v_1)- \sigma(v_2) \right)^2\sum_{k=1}^{\infty}c_k^2|\langle \varphi, e_k\rangle_{H^{1-g,2}}|^2
\\
&\le \left(\|\varphi\|^2_{H^{1-g,2}} \sum_{k=1}^{\infty}c_k^2\right)\left(\sigma(v_1)- \sigma(v_2) \right)^2.
\end{align*}
In a analogous way we can prove that \textbf{(IG4)} holds. 
Finally, \textbf{(IG5)} follows, because
\begin{align*}
\|G(v_1)-G(v_2)\|^2_{L_{HS}(\mathcal{H};\mathbb{L}^2)}
&\le \left( \sigma(v_1)-\sigma(v_2)\right)^2\left(\sum_{k=1}^{\infty}c_k^2 \|e_k\|^2_{H^{1-g,2}}\right)
\le L^2 \left( \sum_{k=1}^{\infty}c_k^2\right) \|v_1-v_2\|^2_{\mathbb{L}^2}.\end{align*}
Notice that in this example we have  $\text{curl}(G(v)h_k)=c_k\sigma(v)\text{curl }e_k$.
\end{example}

\section{Existence of a unique solution to the Navier-Stokes equations \eqref{NS_0}}
\label{v_sec}
In order to prove the existence of a solution of \eqref{vort}, as well as the desired regularity, we need a certain regularity on the solution process $v$ of \eqref{NS_0}. In this Section we remind an existence and uniqueness result concerning system \eqref{NS_0} and then, under stronger assumptions on the regularity of the initial datum and the covariance operator of the noise, we prove 
higher regularity for its solution.

As usual, we project the first equation of \eqref{NS_0} onto the space of divergence free vectors. 
Thus, we get rid of the pressure and we obtain the abstract form of the Navier-Stokes equations
\begin{equation}
\label{NSabs}
\begin{cases}
\displaystyle  {\rm d}v(t)+ \left[Av(t) +B(v(t),v(t))\right]\,{\rm d}t
       = G(v(t))\, {\rm d}W(t), \quad t \in \left[0,T\right]
 \\
\displaystyle v(0)=v_0,
\end{cases}
\end{equation}
We give the following notion of solution.
\begin{definition}
\label{mart_sol2}
A martingale solution to the Navier-Stokes problem \eqref{NSabs} is a triple consisting of a filtered probability space $(\Omega, \mathcal{F}, \{\mathcal{F}_t\}_{t \in \left[0,T\right]}, \mathbb{P})$, an $\{\mathcal{F}_t\}$-adapted cylindrical $\mathcal{H}$-Wiener process $W$ and an $\{\mathcal{F}_t\}$-adapted measurable process $v$, such that
\begin{enumerate}[label=\roman{*}., ref=(\roman{*})]
\item $v:\left[0,T\right] \times \Omega \rightarrow \mathbb{L}^2$ with $\mathbb{P}$-a.e. path  
\begin{equation*}
v(\cdot, \omega) \in C(\left[0,T\right];\mathbb{L}^2) \cap L^2(0,T;H^{1,2});
\end{equation*}
\item for all $z\in C^{\infty}_{\text{sol}}$ and $t \in \left[0,T\right]$ one has $\mathbb{P}$-a.s.
\begin{equation}
\label{weak_NS_bad}
\langle v(t),z \rangle +  
 \int_0^t \langle A  v(s) ,  z\rangle \, {\rm d}s
+\int_0^t  \langle B(v(s),v(s)),  z\rangle\, {\rm d}s
=\langle v_0,z\rangle+\langle \int_0^tG(v(s))\, {\rm d}W(s),z\rangle.
\end{equation}
\end{enumerate}
\end{definition}

The following result holds.
\begin{proposition}
\label{MiR}
Assume that $v_0 \in \mathbb{L}^2$. If assumptions \textbf{(IG1)} and \textbf{(IG3)} are satisfied, then there exists 
a martingale solution to \eqref{NSabs} such that 
\begin{equation}
\label{energyv}
\mathbb{E}\left[ \sup_{0\le t \le T}\|v(t)\|^2_{\mathbb{L}^2}+ \int_0^T \|\nabla v(t)\|^2_{L^2}\, {\rm d}t\right]<\infty.
\end{equation}
Moreover, under \textbf{(IG5)}, pathwise uniqueness holds.
\end{proposition}

\begin{proof}
The existence of a martingale solution, for square summable initial velocity, follows from \cite[Theorem 5.1 and Lemma 7.2]{BrzMot}. 
The hypothesis we made on the covariance operator of the noise are stronger than those made in \cite{BrzMot}. In particular these latter are implied by our assumptions.
\newline
We prove the uniqueness of the solution by means of a rather classical argument (see \cite{Sch1997}).
Let $v_1$ and $v_2$ be two martingale solutions to system \eqref{NSabs} with $v_1(0)=v_2(0)$. 
Let $V=v_1-v_2$.
This difference satisfies the equation
\begin{equation*}
\begin{cases}
{\rm d}V(t)+ \left[AV(t)+ B(v_1(t),v_1(t))-B(v_2(t),v_2(t))\right]\, {\rm d}t=\left[G(v_1(t))-G(v_2(t)) \right]{\rm d}W(t)
\\
V(0)=0
\end{cases}
\end{equation*}
and this is equivalent to 
\begin{equation*}
\begin{cases}
{\rm d}V(t)+ \left[AV(t)+ B(V(t),v_1(t))+B(v_2(t),V(t))\right]\, {\rm d}t=\left[G(v_1(t))-G(v_2(t)) \right]{\rm d}W(t)
\\
V(0)=0
\end{cases}
\end{equation*}
We shall use the It\^o formula for ${\rm d}\left(e^{-\int_0^t \psi(s)\, {\rm d}s}\|V(t)\|^2_{L^2} \right)$, by choosing $\psi$ as
\begin{equation*}
\psi(t)= (a\|\nabla v_1(t)\|^2_{L^2}+L_g^2), \quad t \in \left[0,T\right],
\end{equation*}
where $L_g$ is the Lipschitz constant given in \textbf{(IG5)} and $a$ is a positive constant given later on. Since $v_1\in L^2(0,T;H^{1,2})$, $\psi \in L^1(0,T)$ $\mathbb{P}$-a.s..
For $t \in \left[0,T\right]$, we have 
\begin{align*}
{\rm d}\left(e^{-\int_0^t \psi(s)\, {\rm d}s}\|V(t)\|^2_{L^2} \right)
&=-\psi(t)e^{-\int_0^t \psi(s)\, {\rm d}s}\|V(t)\|^2_{L^2} {\rm d}t
\\
&+e^{-\int_0^t \psi(s)\, {\rm d}s}{\rm d}\|V(t)\|^2_{L^2},
\end{align*}
where the latter differential is given by
\begin{align*}
{\rm d}\|V(t)\|^2_{L^2}
=2&\left[-\langle AV(t),V(t)\rangle - \langle B(V(t), v_1(t)), V(t)\rangle-\langle B(v_2(t),V(t)),V(t)\rangle \right]{\rm d}t
\\
&+2\langle \left[G(v_1(t))-G(v_2(t)) \right]{\rm d}W(t),V(t)\rangle
\\
&+\|G(v_1(t))-G(v_2(t))\|^2_{L_{\text{HS}}(\mathcal{H};\mathbb{L}^2)}.
\end{align*}
For the first term, we get
\begin{equation*}
\langle AV(t),V(t)\rangle=\|\nabla V(t)\|^2_{L^2}.
\end{equation*}
As regards the non linear term, by \eqref{B00} and Gagliardo-Nierenberg interpolation's inequality we get
\begin{align*}
\langle B(V,v_1),V\rangle+\langle B(v_2,V), V\rangle 
&=\langle B(V,v_1),V\rangle
\le \|V\|^2_{L^4} \|\nabla v_1\|_{L^2}
\le C \|V\|_{L^2}\|\nabla V\|_{L^2}\|\nabla v_1\|_{L^2}.
\end{align*}
By Young inequality, we can infer that for all $\varepsilon >0$ there exists a constant $C_{\varepsilon}>0$ such that 
\begin{equation*}
2\langle B(V,v_1),V\rangle
\le \varepsilon \|\nabla V\|^2_{L^2}+C_{\varepsilon}\|\nabla v_1\|^2_{L^2} \|V\|_{L^2}^2.
\end{equation*}
By \textbf{(IG5)} it follows
\begin{equation*}
\|G(v_1)-G(v_2)\|^2_{L_{\text{HS}}(\mathcal{H};\mathbb{L}^2)}
\le  L^2_g\|V\|^2_{L^2}.
\end{equation*}
So we get
\begin{align*}
{\rm d}\|V(t)\|^2_{L^2} 
&\le (\varepsilon-2) \|\nabla V(t)\|^2_{L^2}
+ \left(C_{\varepsilon}\|\nabla v_1(t)\|^2_{L^2}+L_g^2\right) \|V(t)\|^2_{L^2}
\\
&+\langle\left[G(v_1(t))-G(v_2(t)) \right]{\rm d}W(t),V(t) \rangle.
\end{align*}
Putting $a:=C_{\varepsilon}$, we obtain 
\begin{align*}
{\rm d}\left(e^{-\int_0^t \psi(s)\, {\rm d}s}\|V(t)\|^2_{L^2} \right)
&\le (\varepsilon-2)e^{-\int_0^t \psi(s)\, {\rm d}s}\|\nabla V(t)\|^2_{L^2} 
\\
&+e^{-\int_0^t \psi(s)\, {\rm d}s}\langle\left[G(v_1(t))-G(v_2(t)) \right]{\rm d}W(t),V(t) \rangle.
\end{align*}
Integrating in both sides we get
\begin{align}
\label{Schm}
e^{-\int_0^t \psi(s)\, {\rm d}s}\|V(t)\|^2_{L^2} 
&+(2-\varepsilon)\int_0^te^{-\int_0^r \psi(s)\, {\rm d}s}\|\nabla V(r)\|^2_{L^2} \, {\rm d}r
\notag\\
&\le \int_0^te^{-\int_0^r \psi(s)\, {\rm d}s}\langle\left[G(v_1(r))-G(v_2(r)) \right]{\rm d}W(r),V(r) \rangle.
\end{align}
Let us choose $0<\varepsilon <2$, then by \eqref{Schm} we have
\begin{align*}
e^{-\int_0^t \psi(s)\, {\rm d}s}\|V(t)\|^2_{L^2} 
&\le \int_0^te^{-\int_0^r \psi(s)\, {\rm d}s}\langle\left[G(v_1(r))-G(v_2(r)) \right]{\rm d}W(r),V(r) \rangle.
\end{align*}
Since the r.h.s. is a square integrable martingale, taking the expectation in both members we get
\begin{equation*}
\mathbb{E} \left[e^{-\int_0^t \psi(s)\, {\rm d}s}\|V(t)\|^2_{L^2}  \right] \le 0, \qquad \forall t \in \left[0,T\right].
\end{equation*}
Thus in particular, for any $t \in \left[0,T\right]$
\begin{equation*}
e^{-\int_0^t \psi(s)\, {\rm d}s}\|V(t)\|^2_{L^2}=0, \qquad \mathbb{P}-a.s.. 
\end{equation*}
Thus, if we take a sequence $\{t_k\}_{k=1}^{\infty}$, which is dense in $\left[0,T\right]$, we have
\begin{equation*}
\mathbb{P}\{\|V(t_k)\|_{L^2}=0 \ \text{for all} \ k \in \mathbb{N}\}=1.
\end{equation*}
Since each path of the process $V$ belongs to $C(\left[0,T\right];\mathbb{L}^2)$, we get
\begin{equation*}
\mathbb{P} \{v_1(t)=v_2(t) \ \text{for all} \ t \in \left[0,T\right]\}=1
\end{equation*}
and the proof is complete.
\end{proof}

In particular, pathwise uniqueness and existence of martingale solutions implies existence of a strong solution (see e.g \cite{Ikeda1989}). 
\newline
Here we improve the regularity of the solution under stronger assumptions on the regularity of the initial datum and the covariance operator.

\begin{proposition}
\label{Lqvpro}
Let $q>2$ and
assume that conditions \textbf{(IG1)}, \textbf{(IG2)}, \textbf{(IG3)} and \textbf{(IG5)} hold. 
If  $v_0 \in \mathbb{L}^2\cap\mathbb{L}^q$, then the unique strong solution $v$ to \eqref{NSabs}, in addition to \eqref{energyv}, satisfies for every $1\le p<\infty$
\begin{equation}
\label{Lqv}
\mathbb{E} \sup_{0\le t \le T}\|v(t)\|^p_{\mathbb{L}^q}<C,
\end{equation}
for a positive constant $C$, depending on $q$, $T$, $\|v_0\|_{\mathbb{L}^q}$ and $C_{g,q}$. 
\end{proposition}
\begin{proof}
The proof of existence of solutions requires some Galerkin approximation $v^n$ of $v$, for which a priori estimates are proved uniformly in $n$. Then, by a tightness argument one can pass to the limit proving the existence of a solution. Bearing in mind the existence and uniqueness result given by Proposition \ref{MiR}, we just compute the needed $\mathbb{L}^q$-estimates in order to get \eqref{Lqv}.

Let $q\ge2$ and $p\ge q$. Applying It\^o formula to the function $\|\cdot\|^p_{\mathbb{L}^q}$, 
for all $t \in \left[0,T\right]$ we get 
\begin{align}
\label{stimLqv}
\|v(t)\|_{\mathbb{L}^q}^p 
&\le \|v_0\|^p_{\mathbb{L}^q} + p
\int_0^t \|v(s)\|^{p-q}_{\mathbb{L}^q} \langle |v(s)|^{q-2} v(s), \left[-Av(s)-B(v(s),v(s))\right]\rangle\, {\rm d}s
\notag\\
& \quad + p \int_0^t \|v(s)\|^{p-q}_{\mathbb{L}^q} \langle  |v(s)|^{q-2} v(s), G(v(s))\,{\rm d}W(s)\rangle
\notag\\
& \quad + \frac{p(q-1)}{2} \int_0^t \|v(s)\|_{\mathbb{L}^q}^{p-2} \|G(v(s))\|^2_{R(\mathcal{H}; \mathbb{L}^q)}\, {\rm d}s.
\end{align}
Let us estimate separately the various terms appearing in \eqref{stimLqv}. By the integration by parts formula we get
\begin{equation*}
-\langle |v(s)|^{q-2} v(s), Av(s)\rangle
=-\||v(s)|^{\frac{q-2}{2}}\nabla v(s)\|^2_{L^2}
-(q-2)\int |v(s,x)|^{q-4}|\sum_j v_j(s,x)\nabla v_j(s,x)|^2 dx\le 0,
\end{equation*}
and by \eqref{Bq} 
\begin{equation*}
\langle |v(s)|^{q-2} v(s), B(v(s),v(s))\rangle=0.
\end{equation*}
By the Burkholder-Davis-Gundy inequality we get 
\begin{align*}
\mathbb{E}\sup_{0 \le t \le T}&\left|p \int_0^t \|v(s)\|^{p-q}_{\mathbb{L}^q} \langle  |v(s)|^{q-2} v(s), G(v(s))\,{\rm d}W(s)\rangle\right|^2
\\
&\le C_p \mathbb{E} \int_0^T \|v(s)\|^{2(p-q)}_{\mathbb{L}^q}   \|v(s)\|_{\mathbb{L}^q}^{2(q-1)} \| G(v(s))\|^2_{R(\mathcal{H}, \mathbb{L}^q)}\,{\rm d}s
\\
&= C_p \mathbb{E} \int_0^T   \|v(s)\|_{\mathbb{L}^q}^{2(p-1)} \| G(v(s))\|^2_{R(\mathcal{H}, \mathbb{L}^q)}\,{\rm d}s.
\end{align*}
Therefore, squaring both sides of \eqref{stimLqv} and then taking the expectation of the sup (in time) norm at first, then using Young inequality we get
\begin{align*}
\mathbb{E}\sup_{0\le t \le T}\|v(t)\|_{\mathbb{L}^q}^{2p} 
&\le \|v_0\|^{2p}_{\mathbb{L}^q} 
+ \frac{p^2(q-1)^2}{4} \mathbb{E} \sup_{0\le t \le T}\left|\int_0^t   \|v(s)\|_{\mathbb{L}^q}^{p-2} \| G(v(s))\|^2_{R(\mathcal{H}, \mathbb{L}^q)}\,{\rm d}s\right|^2
\\
&\quad + C_p \mathbb{E} \int_0^T   \|v(s)\|_{\mathbb{L}^q}^{2(p-1)} \| G(v(s))\|^2_{R(\mathcal{H}, \mathbb{L}^q)}\,{\rm d}s
\\
&\le \|v_0\|^{2p}_{\mathbb{L}^q} 
+ C^1_{p,T} \mathbb{E} \int_0^T  \| G(v(s))\|^{2p}_{R(\mathcal{H}, \mathbb{L}^q)}\,{\rm d}s
+ C^2_{p,T}\mathbb{E} \int_0^T  \|v(s)\|_{\mathbb{L}^q}^{2p}\, {\rm d}s
\\
& \le \|v_0\|^{2p}_{\mathbb{L}^q} 
+ C^1_{p,T} \mathbb{E} \int_0^T  \| G(v(s))\|^{2p}_{R(\mathcal{H}, \mathbb{L}^q)}\,{\rm d}s
+ C^2_{p,T} \int_0^T  \mathbb{E}\sup_{0\le s \le r}\|v(s)\|_{\mathbb{L}^q}^{2p}\, {\rm d}r.
\end{align*}
By Proposition \ref{MiR},  $v(t) \in\mathbb{L}^2$ for every $t \in \left[0,T\right]$;
then, by \textbf{(IG1)} and \textbf{(IG2)}, we get
\begin{equation*}
\mathbb{E} \int_0^T  \| G(v(s))\|^{2p}_{R(\mathcal{H}, \mathbb{L}^q)}\,{\rm d}s \le T(C_{g,q})^{2p},
\end{equation*} 
thus 
\begin{align*}
\mathbb{E}\sup_{0\le t \le T}\|v(t)\|_{\mathbb{L}^q}^{2p} 
& \le \|v_0\|^{2p}_{\mathbb{L}^q} 
+ C^1_{p,q,T} + C^2_{p,T} \int_0^T  \mathbb{E}\sup_{0\le s \le r}\|v(s)\|_{\mathbb{L}^q}^{2p}\, {\rm d}r.
\end{align*}
Using Gronwall lemma we obtain \eqref{Lqv}.

This proves the result for $p\ge q$. Therefore it holds also for smaller values, i.e. $1\le p<q$.
\end{proof}

\section{Existence of a unique solution to the vorticity equations \eqref{vort}}
\label{vor_sec_bad}

We aim at proving that there exists a martingale solution to \eqref{vort}, in the sense of the following definition.

\begin{definition}
\label{mar_sol}
A martingale solution to equation \eqref{vort} is a triple consisting of a filtered probability space 
$(\Omega, \mathcal{F}, \{\mathcal{F}_t\}_{t \in \left[0,T\right]}, \mathbb{P})$, 
an $\{\mathcal{F}_t\}$-adapted cylindrical Wiener process $W$ on $\mathcal{H}$ and 
an $\{\mathcal{F}_t\}$-adapted measurable process $\xi$ such that 
$\xi: \left[0,T\right] \times \Omega \rightarrow L^2$ with $\mathbb{P}$-a.a. paths 
\begin{equation*}
\xi(\cdot, \omega) \in C(\left[0,T\right];L^2),
\end{equation*}
and such that for all $z\in C^{\infty}_{\text{sol}}$ and $t \in \left[0,T\right]$
\begin{equation}
\label{weak_vor}
\langle \xi(t),z \rangle =\langle \xi_0,z\rangle  
+ \int_0^t \langle  \xi(s),\Delta z \rangle\, {\rm d}s
+\int_0^t \langle v(s)\xi(s) , \nabla z\rangle \, {\rm d}s+\langle\int_0^t \tilde G(v(s))\, {\rm d}W(s),z\rangle
\end{equation}
$\mathbb{P}$-a.s., where $v$ is the solution to \eqref{NSabs}.
\end{definition}
The regularity of the paths of this solution and the regularity of $v$ proved in Proposition \ref{Lqvpro} makes all the terms in \eqref{weak_vor} well defined. 
The well posedness of the stochastic term follows from \eqref{wpstoc}. As regard the well posedness of the non linear term, from \eqref{F3} and the Gagliardo-Nirenberg inequality we get that 
\begin{equation*}
|\langle v(s) \xi(s), \nabla z\rangle| \le 
\|v(s)\|_{\mathbb{L}^4} \|\xi(s)\|_{L^2}\|\nabla z\|_{L^4}
\le C
\|v(s)\|^{\frac 12}_{\mathbb{L}^2}\|\nabla v(s)\|^{\frac12}_{L^2} \|\xi(s)\|_{L^2}\|\nabla z\|_{L^4}
\end{equation*}
and the r.h.s. is bounded thanks to \eqref{energyv} and the regularity required for $\xi$.

In order to prove the existence of a martingale solution to problem \eqref{vort} we cannot use It\^o calculus in the spaces $L^2 \cap L^q$, $q\ge 2$, since the covariance of the noise is not regular enough. Following the idea of \cite{BrzFer} we introduce an approximation system by regularizing the covariance of the noise: we shall use the Hille-Yosida approximations. In this way we construct a sequence of approximating processes $\{\xi_n\}_n$ and $\{v_n\}_n$. In order to pass to the limit, as $n \rightarrow \infty$, we shall exploit the tightness of the sequence of their laws. This is obtained working pathwise with two auxiliary processes $\beta_n$ and $\zeta_n$ with $\xi_n= \beta_n + \zeta_n$.
\newline
Thus, we introduce the smoother problems which approximate \eqref{NS_0} and \eqref{vort}, then we prove the tightness of the sequence of the laws and finally we show the convergence. In this way we prove the existence of a martingale solution to \eqref{vort}. 

\subsection{The approximating equation}
\label{App_Sec}
Let us introduce the Hille-Yosida approximations
\begin{equation*}
R_n=n(nI+A)^{-1}, \qquad n=1,2,...
\end{equation*}
and let us define the approximation sequence
\begin{equation*}
 G_n= R_n G, \qquad n=1,2,... 
\end{equation*}
Every $R_n$ is a contraction operator in $H^{s,q}$ and it converges strongly to the identity operator, i.e. (see \cite[Section 1.3]{Pazy1983})
\begin{equation}
\|R_n\|_{L(H^{s,q},H^{s,q})} \le 1 \quad \text{and} \quad \lim_{n\rightarrow \infty} R_nh=h,\quad \forall \ h \in H^{s,q}.
\end{equation}
Moreover, each $R_n$ is a bounded operator from $H^{s,q}$ to $H^{s+t,q}$ for any $t \le 2$, but the operator norm is not uniformly bounded in $n$ for $t>0$ (for the details see \cite[Section 3.1]{BrzFer}).
From the above and \cite{Baxendale} 
\begin{equation}\label{Gn}
\| G_n(v)\|_{R(\mathcal{H};H^{1-g,q})} \le \| G(v)\|_{R(\mathcal{H};H^{1-g,q})}, \qquad \forall n
\end{equation}
and
\begin{equation}
\lim_{n \rightarrow \infty} \| G_n(v) -  G(v)\|_{R(\mathcal{H};H^{1-g,q})}=0.
\end{equation}
The operator $ G_n(v)$ is more regular than $G(v)$. Indeed, assuming \textbf{(IG1)} and \textbf{(IG2)} (or \textbf{(IG3)}), $ G_n(v)$ is a $\gamma$-radonifying operator in $H^{1,q}$, $q\ge 2$. In fact, for $g \in (0,1)$
\begin{align}
\label{Gn_reg}
\|G_n(v)\|_{R(\mathcal{H};H^{1,q})} 
&\le \|R_nJ^g\|_{\mathcal{L}(H^{1,q}, H^{1,q})}\|J^{-g} G (v)\|_{R(\mathcal{H};H^{1,q})}
\notag\\
&\le \|R_n\|_{\mathcal{L}(H^{1,q},H^{1+g,q})}\| G(v)\|_{R(\mathcal{H},H^{1-g,q})}.
\end{align}
For every $n \in \mathbb{N}$ we consider the approximating problem 
\begin{equation}
\label{NSabs_apx}
\begin{cases}
\displaystyle  {\rm d}v(t)+ \left[Av(t) +B(v(t),v(t))\right]\,{\rm d}t
       = G_n(v(t))\, {\rm d}W(t), \quad t \in \left[0,T\right]
 \\
\displaystyle v(0)=v_0
\end{cases}
\end{equation}

By taking the \textit{curl} on both sides of the first equation we obtain the approximating equation for the vorticity:
\begin{equation}
\label{vort_ap}
\begin{cases}
\displaystyle  {\rm d}\xi(t)+\left[A \xi(t)+v(t) \cdot \nabla \xi(t)\right]{\rm d}t
       = \tilde G_n(v(t))\, {\rm d}W(t), \qquad t \in \left[0,T\right]
\\ 
\xi= \nabla^{\perp}\cdot v \\
\displaystyle \xi(0,x)=\xi_0(x)
\end{cases}
\end{equation}
With the same abuse of notation used above, for every $n \in \mathbb{N}$, we write $\tilde G_n(v){\rm d}W(t)$ instead of curl$(G_n(v){\rm d}W(t))$, where 
$\tilde G_n:= \text{curl}G_n$.
This is the vorticity equation \eqref{vort} with a more regular noise. 

The next result provides the existence of a unique strong solution to system \eqref{vort_ap}, for any fixed $n \in \mathbb{N}$. We recall that by strong solution to \eqref{vort_ap} we mean an $\{\mathcal{F}_t\}$-adapted measurable process $\xi$ such that $\xi: \left[0,T\right] \times \Omega \rightarrow L^2$  with $\mathbb{P}$-a.s. paths 
$\xi(\cdot, \omega) \in C(\left[0,T\right];L^2)
$, that satisfies \eqref{weak_vor}, where the last term is replaced by $\langle\int_0^t \tilde G_n(v(s))\, {\rm d}W(s),z\rangle$. Here the stochastic basis $(\Omega, \mathcal{F}, \{\mathcal{F}_t\}_{t \in \left[0,T\right]}, \mathbb{P})$ is given in advance and it is not constructed as a part of the solution. The proof of Proposition \ref{prop3} is based on a more general result whose statement and proof are postponed to Appendix \ref{A}.
\begin{proposition}
\label{prop3}
Assume conditions \textbf{(IG1)}, \textbf{(IG4)} and \textbf{(IG5)}. Let $\xi_0\in L^2$ and $v_0 \in \mathbb{L}^2$. Then, for each $n\in \mathbb{N}$, there exists a unique strong solution $\xi_n$ to \eqref{vort_ap}. Moreover,  
\begin{equation*}
\xi_n \in L^p(\Omega;L^{\infty}(0,T;L^2)) \cap L^2(\Omega; L^2(0,T;W^{1,2})), \qquad 
\forall p>1
\end{equation*}
and there exists a constant $C_n$ such that
\begin{equation}
\label{xi_notuni}
\mathbb{E} \left[ \sup_{0\le t \le T}\|\xi_n(t)\|^p_{L^2}\right] + \mathbb{E} \left[\int_0^T \|\xi_n(t)\|^2_{W^{1,2}}\, {\rm d}t \right] \le C_n.
\end{equation}
\end{proposition}
\begin{proof}
Thanks to \eqref{Gn_reg}, the operator $G_n$ is regular enough to apply Theorem \ref{existence} 
and infer, for any $n \in \mathbb{N}$, the existence of a martingale solution (in the sense of Definition \ref{mart_sol}) to \eqref{NSabs_apx}. Moreover, under assumption \textbf{(IG5)}, the solution is pathwise unique. 
Thus \eqref{NSabs_apx} admits a unique strong solution. As a consequence we infer that, for any $n \in \mathbb{N}$, there exists a strong solution 
 of the approximating problem \eqref{vort_ap}. 
 This is obtained by taking the \textit{curl} of the solution to equation \eqref{NSabs_apx}. In particular, from \eqref{blu} we infer \eqref{xi_notuni}.
 \end{proof}

\subsection{Tightness of the law of $\{v_n\}_n$}
\label{tight_sec_v}
In this Section we provide the tightness of the sequence of the laws of $\{v_n\}_n$ in proper spaces. The crucial point is to obtain uniform estimates in $n \in \mathbb{N}$.

\begin{proposition}
\label{v_n_prop}
Assume \textbf{(IG1)}, \textbf{(IG3)} and \textbf{(IG5)}.
If $v_0 \in \mathbb{L}^2$, then there exists a unique strong solution to \eqref{NSabs_apx}
 for each $n \in \mathbb{N}$.
\newline
Moreover, 
\begin{equation}
\label{energyv_n}
\sup_{n \in \mathbb{N}}\mathbb{E}\left[ \sup_{0\le t \le T}\|v_n(t)\|^2_{\mathbb{L}^2}+ \int_0^T \|\nabla v_n(t)\|^2_{L^2}\, {\rm d}t\right]<\infty.
\end{equation}
In particular, for any $\varepsilon>0$ there exist positive constants $\alpha_i$, $i=1,2,3$ such that
\begin{equation}
\label{rem_vel_eq}
\sup_n \mathbb{P} \left(\|v_n\|_{L^{\infty}(0,T;\mathbb{L}^2)}>\alpha_1 \right) \le \varepsilon,
\end{equation}
\begin{equation}
\label{rem_vel_eq3}
\sup_n \mathbb{P} \left(\|v_n\|_{L^2(0,T;H^{1,2})}>\alpha_2 \right) \le \varepsilon.
\end{equation}
\begin{equation}
\label{Rem_vel_eqL4}
\sup_n \mathbb{P} \left(\|v_n\|_{L^4(0,T;\mathbb{L}^4)}>\alpha_3 \right) \le \varepsilon.
\end{equation}
Moreover, there exists $\mu>0$ such that for any $\varepsilon >0$ there exists a positive constant $\alpha_4$ such that
\begin{equation}
\label{9a}
\sup_n \mathbb{P} \left(\|v_n\|_{C^{\mu}(\left[0,T\right];H^{-1,2})}>\alpha_4 \right) \le \varepsilon.
\end{equation}
\end{proposition}
\begin{proof}
The proof of \eqref{energyv_n} immediately follows from the results of Section \ref{v_sec}. Indeed, by \eqref{Gn} we get a uniform estimate on $G_n(v)$.
From this we infer the estimates in probability  \eqref{rem_vel_eq} and \eqref{rem_vel_eq3}, which in 
turn imply \eqref{Rem_vel_eqL4} thanks to the Gagliardo-Nirenberg inequality 
$\|v_n(s)\|_{\mathbb L^4}\le C \|v_n(s)\|_{\mathbb L^2}^{1/2} \|\nabla v_n(s)\|^{1/2}_{L^2}$.

Finally, estimate \eqref{9a} comes from Proposition 3.5 of \cite{BrzFer}. Indeed, all the assumptions
of that Proposition are fulfilled; in particular  
the continuous embedding $H^{1-g,2}\subset H^{-g,4}$ implies assumption (G2) 
of Proposition 3.5 in \cite{BrzFer}.
\end{proof}
In the same way, from Proposition \ref{Lqvpro} we get 
\begin{proposition}
\label{v_n_prop2}
Let  $q > 2$ and assume \textbf{(IG1)}, \textbf{(IG2)}, \textbf{(IG3)} and \textbf{(IG5)}.
Let $v_0 \in \mathbb{L}^2 \cap \mathbb{L}^q$. Let $\{v_n\}$ be the solution to \eqref{NSabs_apx} as given in Proposition \ref{v_n_prop}. Then, in addition to \eqref{energyv_n}-\eqref{9a}, for any $1< p< \infty$ it holds, 
\begin{equation}
\label{supnv}
\sup_{n \in \mathbb{N}}\ \mathbb{E}\sup_{0 \le t \le T} \|v_n(t)\|^p_{\mathbb{L}^q} < \infty.
\end{equation}
In particular, for any $\varepsilon>0$ there exists a positive constant $\alpha_4$, such that
\begin{equation}
\label{rem_vel_eq2}
\sup_n \mathbb{P} \left(\|v_n\|_{L^{\infty}(0,T;\mathbb{L}^q)}>\alpha_4 \right) \le \varepsilon.
\end{equation}
\end{proposition}

\subsection{Tightness of the law of $\xi_n$}
\label{tight_sec_xi}
The present Section is devoted to the proof of the tightness of the sequence of the laws of $\{\xi_n\}_n$. Let us start by noticing that estimate \eqref{xi_notuni} is not uniform with respect to $n$: \eqref{Gn_reg} shows that the 
$\gamma$-radonifying norms of the $G_n(v)$ (and thus of the $\tilde G_n(v)$) are not uniformly bounded in $n$. 
Therefore, from \eqref{xi_notuni} we cannot obtain the tightness of the sequence of the laws of the 
$\xi_n$'s. 
In order to get uniform estimates in $n$ for the sequence $\{\xi_n\}_n$, 
we follow the idea of \cite{BrzFer}, splitting our problem in two subproblems in the unknowns $\zeta_n$ and $\beta_n$ with $\xi_n=\zeta_n+ \beta_n$.

We define the process $\zeta_n$ as the solution of the Ornstein-Uhlenbeck equation
\begin{equation}
\begin{cases}
{\rm d} \zeta_n(t)+ A\zeta_n(t)\ { \rm d}t= \tilde G_n(v_n(t))\,{\rm d}W(t), \qquad t \in \left[0,T\right]
\\
\zeta_n(0)=0.
\end{cases}
\end{equation}
Therefore, the process $\beta_n= \xi_n -\zeta_n$ solves 
\begin{equation}
\label{beta_n_eq}
\begin{cases}
\frac{\rm d \beta_n}{{\rm d t}}(t)+ A\beta_n(t) +v_n(t) \cdot \nabla \xi_n(t)= 0, \qquad t \in \left[0,T\right]
\\
\beta_n(0)=\xi_0.
\end{cases}
\end{equation}
We shall first analyze the Ornstein-Uhlenbeck processes; the solution 
 $\zeta_n$ is given by
\begin{equation}
\zeta_n(t)=\int_0^t S(t-s)\tilde G_n(v_n(s))\, {\rm d}W(s).
\end{equation}
With a slight modification of the proofs of \cite[Lemma 3.2]{BrzFer} and \cite[Lemma 3.3]{BrzFer} respectively, we have the following regularity results. Recall that assumptions \textbf{(IG1)}-\textbf{(IG2)} reads as \textbf{(I$\tilde G$1)}-\textbf{(I$\tilde G$2)} when we deal with the equation for the vorticity.

\begin{lemma}
\label{3.2}
Let $q \ge 2$. Assume conditions \textbf{(IG1)} and \textbf{(IG2)}. Take any $g_0 \in \left[g,1\right)$ and put $\varepsilon=g_0-g \ge 0$. Then, for any integer $m \ge 2$ there exists a constant $C$ independent of $n$ (but depending on $m$, $T$, $q$, $g_0$ and $\tilde C_{g,q}$) such that 
\begin{equation*}
\mathbb{E}\|\zeta_n\|^m_{L^m(0,T;W^{\varepsilon,q})}\le C.
\end{equation*}
In particular, $\zeta_n \in L^m(0,T;W^{\varepsilon,q})$ $\mathbb{P}$-a.s.
\end{lemma}
\begin{lemma}
\label{3.3}
Let $q \ge 2$, assume \textbf{(IG1)} and let 
\begin{equation*}
0 \le \beta < \frac{1-g}{2}.
\end{equation*}
Then for any $p \ge 2$ and $\delta \ge 0$ such that 
\begin{equation*}
\beta + \frac{\delta}{2}+ \frac 1p < \frac{1-g}{2}
\end{equation*}
there exists a modification $\tilde \zeta_n$ of $\zeta_n$ such that 
\begin{equation}
\mathbb{E}\|\tilde \zeta_n\|^p_{C^{\beta}(\left[0,T\right];W^{\delta,q})} \le \tilde C
\end{equation}
for some constant $\tilde C$ independent of $n$ (but depending on $T, \ \beta, \delta, \ p$ and $q$).
\end{lemma}

As a consequence of Lemma \ref{3.2} and Lemma \ref{3.3} we have that there exist finite constants $K_{m,q}$ and $K_{\beta, \delta,q}$ such that 
\begin{equation*}
\sup_n \mathbb{E}\|\zeta_n\|^m_{L^m(0,T;W^{\varepsilon, q})}=(K_{m,q})^m
\end{equation*}
and
\begin{equation*}
\sup_n \mathbb{E}\|\zeta_n\|^p_{C^{\beta}(\left[0,T\right];W^{\delta,q})}=(K_{\beta, \delta,q})^p.
\end{equation*}
Therefore, by Chebychev's inequality,  for any $\eta>0$
\begin{equation}
\label{eta1}
\sup_n \mathbb{P}\left(\|\zeta_n\|_{L^m(0,T;W^{\varepsilon, q})}>\eta\right)\le\frac{(K_{m,q})}{\eta}
\end{equation}
and
\begin{equation}
\label{eta2}
\sup_n \mathbb{P}\left(\|\zeta_n\|_{C^{\beta}(\left[0,T\right];W^{\delta,q})}>\eta\right)\le \frac{(K_{\beta, \delta,q})}{\eta}.
\end{equation}
Thanks to these two last inequalities we get uniform estimates in probability for the sequence $\beta_n$ (see Propositions \ref{beta_n_prop} and \ref{beta_n_prop2}), and consequently for $\xi_n=\beta_n+ \zeta_n$ (see Proposition \ref{xi_n_prop}).
\newline
Let us now turn to the analysis of equation \eqref{beta_n_eq}. We shall analyze it pathwise, proving the following result.

\begin{proposition}
\label{beta_n_prop}
Let $q=4$ and assume \textbf{(IG1)} and \textbf{(IG2)}. Let $\xi_0 \in L^2$ and $v_0 \in \mathbb{L}^2$.
Then for every $n\in \mathbb{N}$ the paths of the process $\beta_n= \xi_n - \zeta_n$ solving \eqref{beta_n_eq} are such that
\begin{equation*}
\beta_ n \in C(\left[0,T\right];L^2) \cap L^{4}(0,T; L^4)\cap L^2(0,T;W^{1,2}) \cap C^{\frac 12}(\left[0,T\right];W^{-1,2})
\end{equation*}
$\mathbb{P}$-a.s., and for any $\varepsilon >0$ there exist constants $C_i=C_i(\varepsilon)$, $i=1,...4$ such that 
\begin{equation}
\label{betaC1}
\sup_n \mathbb{P} \left(\|\beta_n\|_{L^{\infty}(0,T;L^2)}>C_1 \right) \le \varepsilon
\end{equation}
\begin{equation}
\label{betaC3}
\sup_n \mathbb{P} \left(\|\beta_n\|_{L^2(0,T;W^{1,2})}>C_2 \right) \le \varepsilon
\end{equation}
\begin{equation}
\label{betaC2}
\sup_n \mathbb{P} \left(\|\beta_n\|_{L^{4}(0,T;L^4)}>C_3 \right) \le \varepsilon
\end{equation}
\begin{equation}
\label{betaC4}
\sup_n \mathbb{P} \left(\|\beta_n\|_{C^{\frac 12}(\left[0,T\right];W^{-1,2})}>C_4 \right) \le \varepsilon.
\end{equation}
\end{proposition}
\begin{proof}
By definition and merging the regularity of $\xi_n$ and $\zeta_n$ we have that $\beta_n \in L^{\infty}(0,T;L^2)$.
\newline
Let us prove estimates \eqref{betaC1}-\eqref{betaC4}. We begin with the usual energy estimate
 
\begin{equation}
\label{enest}
\frac 12 \frac{\rm d}{{\rm d}t} \|\beta_n(t)\|^2_{L^2}+ \|\nabla \beta_n(t)\|_{L^2}^2= - \langle  v_n(t) \cdot \nabla \xi_n(t), \beta_n(t)\rangle.
\end{equation}
Let us focus on the trilinear term. By means of \eqref{F2}, Gagliardo-Nirenberg and Young's inequalities, and using \eqref{equivnorm} we get
\begin{align*}
- \langle  v_n(t) \cdot \nabla \xi_n(t), \beta_n(t)\rangle
&=\langle v_n(t) \cdot \nabla \beta_n(t), \xi_n(t)\rangle
\\
&=\langle v_n(t) \cdot \nabla \beta_n(t), \beta_n(t)\rangle+ \langle v_n(t) \cdot \nabla \beta_n(t), \zeta_n(t)\rangle
\\
&=\langle v_n(t) \cdot \nabla \beta_n(t), \zeta_n(t)\rangle
\\
&\le \|\nabla \beta_n(t)\|_{L^2} \|v_n(t)\|_{\mathbb{L}^4} \|\zeta_n(t)\|_{L^4}
\\
&\le C\|\nabla \beta_n(t)\|_{L^2}\|\zeta_n(t)\|_{L^4} \|v_n(t)\|^{\frac 12}_{\mathbb{L}^2} \|\nabla v_n(t)\|_{\mathbb{L}^2}^{\frac 12}
\\
&\le C\|\nabla \beta_n(t)\|_{L^2}\|\zeta_n(t)\|_{L^4} \left(\frac 12\|v_n(t)\|_{\mathbb{L}^2}+ \frac 12 \|\nabla v_n(t)\|_{\mathbb{L}^2}\right)
\\
&\le C\|\nabla \beta_n(t)\|_{L^2}\|\zeta_n(t)\|_{L^4} \left(\frac 12\|v_n(t)\|_{\mathbb{L}^2}+ \frac 12 \|\xi_n(t)\|_{L^2}\right)
\\
&\le \frac 14 \|\nabla \beta_n(t)\|^2_{L^2} + C \|\zeta_n(t)\|^2_{L^4} \|v_n(t)\|^2_{\mathbb{L}^2}+ \frac 14 \|\nabla \beta_n(t)\|^2_{L^2} 
\\
& \qquad \quad+C \|\zeta_n(t)\|^2_{L^4} \|\beta_n(t)\|^2_{L^2}+C \|\zeta_n(t)\|^2_{L^4} \|\zeta_n(t)\|^2_{L^2}
\\
&=\frac 12\|\nabla \beta_n(t)\|^2_{L^2}+C_1 \|\zeta_n(t)\|^2_{L^4}\left(\|v_n(t)\|^2_{\mathbb{L}^2}+ \|\zeta_n(t)\|^2_{L^2}\right)
\\
&\qquad  + \frac {C_2}2 \|\zeta_n(t)\|^2_{L^4} \|\beta_n(t)\|^2_{L^2}.
\end{align*}
Let us set 
\begin{equation}
\psi_n (t):=2 C_1 \|\zeta_n(t)\|^2_{L^4}\left(\|v_n(t)\|^2_{\mathbb{L}^2}+ \|\zeta_n(t)\|^2_{L^2}\right).
\end{equation}
Using H\"older's inequality, by Lemma \ref{3.2} and Proposition \ref{v_n_prop} 
we have that $\psi_n \in L^1(0,T)$.

Then from \eqref{enest} we get
\begin{align}
\label{enest2}
\frac{\rm d}{{\rm d}t} \|\beta_n(t)\|^2_{L^2}+ \|\nabla \beta_n(t)\|_{L^2}^2
\le \psi_n (t)+ C_2 \|\zeta_n(t)\|^2_{L^4} \|\beta_n(t)\|^2_{L^2}.
\end{align}

Hence from Gronwall's lemma applied to inequality
\begin{align*}
\frac{\rm d}{{\rm d}t} \|\beta_n(t)\|^2_{L^2}
\le \psi_n (t)+ C_2 \|\zeta_n(t)\|^2_{L^4} \|\beta_n(t)\|^2_{L^2}
\end{align*}
we infer that 
\begin{align}\label{C7}
\sup_{0\le t \le T} \|\beta_n(t)\|_{L^2}^2 
 &\le \|\xi_0\|^2_{L^2}e^{C_2\int_0^T\|\zeta_n(r)\|^2_{L^4}\, {\rm d}r} + e^{C_2\int_0^T\|\zeta_n(r)\|^2_{L^4}\, {\rm d}r}\int_0^T  \psi_n(s)\,{\rm d}s.
\end{align}
Then, integrating in time estimate \eqref{enest2} we  
infer that for all $n$
\begin{equation}\label{C6}
\begin{split}
\int_0^T\|\nabla \beta_n(t)\|^2_{L^2}\, {\rm d}t 
&\le \|\xi_0\|^2_{L^2}+ \int_0^T \left(\psi_n(t)+C_2 \|\zeta_n(t)\|^2_{L^4} \|\beta_n(t)\|^2_{L^2}\right)\, {\rm d}t
\\
&\le  \|\xi_0\|^2_{L^2}+ C_2\left(\sup_{0 \le t \le T} \|\beta_n(t)\|^2_{L^2}\right)\int_0^T  \|\zeta_n(t)\|^2_{L^4}\, {\rm d}t +  \int_0^T \psi_n(t)\, {\rm d}t.
\end{split}
\end{equation}
Recalling \eqref{rem_vel_eq} and \eqref{eta1}, we infer that for any $\varepsilon >0$ there exists a constant $C_7$ such that 
\begin{equation*}
\sup_n \mathbb{P}\left(\int_0^T \psi_n(t)\,{\rm d}t>C_7\right) \le \varepsilon.
\end{equation*}
Therefore, from \eqref{C7} and \eqref{C6} we get that, for any $\varepsilon >0$ there exist suitable constants $R_1, R_2>0$ such that 
\begin{equation*}
\sup_n \mathbb{P} \left(\|\beta_n\|_{L^{\infty}(0,T;L^2)}>R_1 \right) \le \varepsilon, 
\qquad 
\sup_n \mathbb{P} \left(\|\nabla \beta_n\|_{L^2(0,T;L^2)}>R_2 \right) \le \varepsilon.
\end{equation*}
From these two last inequalities  it is straightforward to see that, for any $\varepsilon>0$, there exists a suitable constant $R_3>0$ such that 
\begin{equation*}
\sup_n \mathbb{P} \left(\|\beta_n\|_{L^2(0,T;W^{1,2})}>R_3 \right) \le \varepsilon.
\end{equation*}
These estimates prove \eqref{betaC1} and \eqref{betaC3}.

Now, as done in Proposition \ref{v_n_prop} we obtain \eqref{betaC2} from 
\eqref{betaC1} and \eqref{betaC3} by means of 
Gagliardo-Nirenberg inequality.

Finally,  from \eqref{beta_n_eq} we infer
\begin{equation*}
\left\Vert \frac{{\rm d}\beta_n}{{\rm d}t}\right\Vert_{L^2(0,T;W^{-1,2})}
\le \|A\beta_n\|_{L^2(0,T;W^{-1,2})} + \|v_n \cdot \nabla \xi_n\|_{L^2(0,T;W^{-1,2})}.
\end{equation*}
Bearing in mind \eqref{F3}  we get
\begin{align*}
\|v_n \cdot \nabla \xi_n\|_{L^2(0,T;W^{-1,2})} 
&\le \left(\int_0^T \|v_n(t)\|^2_{\mathbb{L}^4}\|\xi_n(t)\|^2_{L^4} \, {\rm d}t\right)^{\frac12}
\\
&\le \|v_n\|_{L^4(0,T;\mathbb{L}^4)}\|\xi_n\|_{L^4(0,T;L^4)}
\\
&\le \|v_n\|^2_{L^4(0,T;\mathbb{L}^4)}+\|\beta_n\|^2_{L^4(0,T;L^4)}+\|\zeta_n\|^2_{L^4(0,T;L^4)}
\end{align*}
Thus,
\begin{align*}
\left\Vert \frac{{\rm d}\beta_n}{{\rm d}t}\right\Vert_{L^2(0,T;W^{-1,2})}
&\le \|\beta_n\|_{L^2(0,T;W^{1,2})} 
+\|v_n\|^2_{L^4(0,T;\mathbb{L}^4)}+\|\beta_n\|^2_{L^4(0,T;L^4)}+\|\zeta_n\|^2_{L^4(0,T;L^4)}.
\end{align*}
From \eqref{betaC3}, \eqref{Rem_vel_eqL4}, \eqref{betaC2} and \eqref{eta1}
 we  find that for any $\varepsilon >0$, there exists a suitable constant $R_4>0$ such that 
\begin{equation*}
\sup_n\mathbb{P}\left(\left\Vert \frac{{\rm d}\beta_n}{{\rm d}t}\right\Vert_{L^2(0,T;W^{-1,2})}>R_4 \right) \le \varepsilon.
\end{equation*}

Now we recall the Sobolev's embedding Theorem $H^{1,2}(0,T)=\{u \in L^2(0,T): u' \in L^2(0,T)\} \subset C^{\frac 12}(\left[0,T\right])$. Hence, there exists a constant $R_5>0$ such that 
\begin{equation*}
\sup_n\mathbb{P}\left(\left\Vert\beta_n\right\Vert_{C^{\frac 12}(\left[0,T\right];W^{-1,2})}>R_5 \right) \le \varepsilon,
\end{equation*}
which proves \eqref{betaC4}. 

We conclude proving the continuity in time. From the previous estimates we have that
$\frac{{\rm d}\beta_n}{{\rm d}t}\in L^2(0,T;W^{-1,2})$ and $\beta_n \in L^2(0,T;W^{1,2})$. 
Therefore (see \cite[Theorem III.1.2]{Temam2001}) we get $\beta_n \in C([0,T];L^2)$.
\end{proof}

\begin{proposition}
\label{beta_n_prop2}
Assume that conditions of Proposition \ref{beta_n_prop} hold. In addition assume that condition \textbf{(IG2)} holds also for a $q>2$. Let $\xi_0 \in L^2 \cap L^q$ and $v_0 \in \mathbb{L}^2 \cap\mathbb{L}^q$.
Then for every $n\in \mathbb{N}$ the paths of the process $\beta_n= \xi_n - \zeta_n$ solving \eqref{beta_n_eq} are such that
\begin{equation*}
\beta_ n \in C(\left[0,T\right];L^2) \cap L^{\infty}(0,T; L^q)\cap L^{4}(0,T; L^4)
\cap L^2(0,T;W^{1,2}) \cap C^{\frac 12}(\left[0,T\right];W^{-1,2})
\end{equation*}
$\mathbb{P}$-a.s., and, in addition to \eqref{betaC1}-\eqref{betaC4}, for any $\varepsilon >0$ there exist a constant $C_5$, such that 
\begin{equation}
\label{betaC5}
\sup_n \mathbb{P} \left(\|\beta_n\|_{L^{\infty}(0,T;L^q)}>C_5 \right) \le \varepsilon.
\end{equation}
\end{proposition}
\begin{proof}
Let us estimate the $L^q$-norm for $q>2$. Let $x \in \mathbb{R}^2$ and $t \in \left[0,T\right]$. We get
\begin{align*}
\frac{\partial}{\partial t}|\beta_n(t,x)|^q=q|\beta_n(t,x)|^{q-2}\beta_n(t,x)\left( \Delta \beta_n(t,x)-v_n(t,x) \cdot \nabla \xi_n(t,x)\right).
\end{align*}
Integrating on $\mathbb{R}^2$, by means of the integration by parts formula we get
\begin{align}
\label{bq}
\frac{{\rm d}}{{\rm d}t}\|\beta_n(t)\|^q_{L^q} 
&=q \langle |\beta_n(t)|^{q-2}\beta_n(t), \Delta \beta_n(t)\rangle
-q\langle |\beta_n(t)|^{q-2}\beta_n(t)v_n(t), \nabla \xi_n(t)\rangle 
\notag\\
&=-q(q-1)\||\beta_n(t)|^{\frac{q-2}{2}}\nabla \beta_n(t)\|^2_{L^2}
-q\langle |\beta_n(t)|^{q-2}\beta_n(t)v_n(t), \nabla \xi_n(t)\rangle. 
\end{align}
Let us estimate the nonlinear term. Thanks to \eqref{Fq}-\eqref{Fq0} 
we get
\begin{align*}
-q\langle |\beta_n(t)|^{q-2}&\beta_n(t)v_n(t), \nabla \xi_n(t)\rangle 
\\
&=-q\langle |\beta_n(t)|^{q-2}\beta_n(t)v_n(t), \nabla \beta_n(t)\rangle
-q\langle |\beta_n(t)|^{q-2}\beta_n(t)v_n(t), \nabla \zeta_n(t)\rangle
\\
&=q{(q-1)}\langle |\beta_n(t)|^{q-2}\zeta_n(t)v_n(t), \nabla \beta_n(t)\rangle\end{align*}
By means of Young's inequality and \eqref{Sobolev}, recalling \eqref{equivnorm} we get
\begin{align*}
|\langle |\beta_n(t)|^{q-2}&\zeta_n(t)v_n(t), \nabla \beta_n(t)\rangle|
\\
&\le \||\beta_n(t)|^{q-2} \nabla \beta_n(t)\|_{L^2} \|\zeta_n(t)\|_{L^2}\|v_n(t)\|_{L^{\infty}}\\
&\le\frac 12 \||\beta_n(t)|^{q-2} \nabla \beta_n(t)\|^2_{L^2} + \frac C2\|\zeta_n(t)\|^2_{L^2}\|v_n(t)\|^2_{H^{1,q}}
\\
&=\frac 12 \||\beta_n(t)|^{q-2} \nabla \beta_n(t)\|^2_{L^2} + C\|\zeta_n(t)\|^2_{L^2}\left(\|v_n(t)\|^2_{\mathbb{L}^q}+ \|\xi_n(t)\|_{L^q}^2\right)
\\
&\le \frac 12 \||\beta_n(t)|^{q-2} \nabla \beta_n(t)\|^2_{L^2} + C\|\zeta_n(t)\|^2_{L^2}\left(\|v_n(t)\|^2_{\mathbb{L}^q}+ \|\beta_n(t)\|_{L^q}^2+\|\zeta_n(t)\|_{L^q}^2\right)
\\
& \le \frac 12 \||\beta_n(t)|^{q-2} \nabla \beta_n(t)\|^2_{L^2} 
+C_1 \|\beta_n(t)\|^q_{L^q}+ C_2\|\zeta_n(t)\|^{\frac {2q}{q-2}}_{L^2}
\\
&\qquad \qquad + \|\zeta_n(t)\|^2_{L^2}\left(\|v_n(t)\|^2_{\mathbb{L}^q}+\|\zeta_n(t)\|_{L^q}^2 \right)
\end{align*}
Let us set 
\begin{equation*}
\varphi_n(t)=C_2\|\zeta_n(t)\|^{\frac {2q}{q-2}}_{L^2}
+ \|\zeta_n(t)\|^2_{L^2}\left(\|v_n(t)\|^2_{\mathbb{L}^q}+\|\zeta_n(t)\|_{L^q}^2 \right),
\end{equation*}
then from \eqref{bq} we get 
\begin{align}
\frac{{\rm d}}{{\rm d}t}\|\beta_n(t)\|^q_{L^q} 
+\frac{q(q-1)}{2}\||\beta_n(t)|^{q-2}\nabla \beta_n(t)\|^2_{L^2}
\le q(q-1)\left(\varphi_n(t) + C_1\|\beta_n(t)\|^q_{L^q}\right).
\end{align}
Using H\"older's inequality, by Lemma \ref{3.2} and Proposition \ref{v_n_prop2}, we have that $\varphi_n \in L^1(0,T)$ uniformly in $n$. Hence from Gronwall's lemma applied to inequality
\begin{align*}
\frac{{\rm d}}{{\rm d}t}\|\beta_n(t)\|^q_{L^q} 
\le q(q-1)\left( \varphi_n(t) + C_1\|\beta_n(t)\|^q_{L^q}\right),
\end{align*}
we infer that for all $n$
\begin{align}
\label{C8}
\sup_{0\le t \le T} \|\beta_n(t)\|_{L^q}^q 
&\le \|\xi_0\|^q_{L^q}e^{q(q-1)C_1T} + q(q-1)e^{q(q-1)C_1T}\int_0^T \varphi_n(s)\,{\rm d}s.
\end{align}
Recalling \eqref{rem_vel_eq2} and \eqref{eta1}, we infer that for any $\varepsilon >0$ there exists a constant $C_0$ such that 
\begin{equation*}
\sup_n \mathbb{P}\left(\int_0^T \varphi_n(t)\,{\rm d}t>C_0\right) \le \varepsilon.
\end{equation*}
Therefore, from \eqref{C8} we get that, for any $\varepsilon >0$ there exist suitable constant $R_4>0$ such that 
\begin{equation*}
\sup_n \mathbb{P} \left(\|\beta_n\|_{L^{\infty}(0,T;L^q)}>R_4 \right) \le \varepsilon.
\end{equation*}
This proves \eqref{betaC2}.
\end{proof}

In order to pass to the limit we shall now apply a tightness argument. Merging the estimates \eqref{eta1}-\eqref{eta2} for $\zeta_n$ and those for $\beta_n$ in Proposition \ref{beta_n_prop} we get the estimates of $\xi_n=\zeta_n+ \beta_n$. These estimates in probability are uniform with respect to $n$.

\begin{proposition}
\label{xi_n_prop}
\begin{enumerate}[label=\roman{*}), ref=(\roman{*})]
\item Let $q=4$ and assume conditions \textbf{(IG1)}, \textbf{(IG2)}, \textbf{(IG4)} and \textbf{(IG5)}.
Let $\xi_0 \in L^2$ and $v_0 \in \mathbb{L}^2$. Let $\xi_n$ be the solution to \eqref{vort_ap} as given in Proposition \ref{prop3}.
\newline
Then there exist $\gamma, \delta>0$ such that for any $\varepsilon>0$ there exist positive constants $\eta_i$, $i=1,...,4$ such that 
\begin{equation*}
\sup_n \mathbb{P} \left(\|\xi_n\|_{L^{\infty}(0,T;L^2)}>\eta_1 \right) \le \varepsilon
\end{equation*}
\begin{equation*}
\sup_n \mathbb{P} \left(\|\xi_n\|_{L^4(0,T;L^4)}>\eta_2 \right) \le \varepsilon
\end{equation*}
\begin{equation*}
\sup_n \mathbb{P} \left(\|\xi_n\|_{L^2(0,T;W^{\delta,2})}>\eta_3 \right) \le \varepsilon
\end{equation*}
\begin{equation*}
\sup_n \mathbb{P} \left(\|\xi_n\|_{C^{\gamma}(\left[0,T\right];W^{-1,2})}>\eta_4 \right) \le \varepsilon.
\end{equation*}
\newline
\item If in addition we assume that condition \textbf{(IG2)} holds also for a $q>2$ and if $\xi_0 \in L^2 \cap L^q$, $v_0 \in \mathbb{L}^2 \cap \mathbb{L}^q$, then for any $\varepsilon>0$ there exist positive constants $\eta_5$, such that 
\begin{equation*}
\sup_n \mathbb{P} \left(\|\xi_n\|_{L^{\infty}(0,T;L^q)}>\eta_5 \right) \le \varepsilon.
\end{equation*}
\end{enumerate}
\end{proposition}
Let us notice that $\gamma= \min(\beta, \frac 12)$, with $\beta$ and $\gamma$ fulfilling hypothesis in Lemma \ref{3.3}; thus $0 < \gamma < \frac 12$ and $0<\delta<1$.

\subsection{Convergence and existence of a unique strong solution}
In order to pass to the limit we exploit a tightness argument. This requires some technical results. If we proceed as in \cite[Lemma 3.3]{BrzMot} and \cite[Lemma 5.3]{BrzFer}, we get the following compactness result. 
\begin{lemma}
Let $\alpha, q>1$ and define
\begin{equation*}
Z=L^{\alpha}_w(0,T;L^q) \cap C(\left[0,T\right];U') \cap L^2(0,T; L^2_{loc})\cap C(\left[0,T\right];L^2_w).
\end{equation*}
Let $\mathcal{T}$ be the supremum of the corresponding topologies. Then a set $K \subset Z$ is $\mathcal{T}$-relatively compact if the following conditions hold:
\begin{enumerate}[label=\roman{*}., ref=(\roman{*})]
\item $\sup_{f \in K}\|f\|_{L^{\alpha}(0,T;L^q)}< \infty$
\item $\exists \ \gamma >0: \sup_{f \in K} \|f\|_{C^{\gamma}(\left[0,T\right];W^{-1,2})}< \infty$
\item $\exists \ \delta >0: \sup_{f \in K} \|f\|_{L^2(0,T;W^{\delta,2})}< \infty$
\item $\sup_{f \in K}\|f\|_{L^{\infty}(0,T;L^2)}< \infty$
\end{enumerate}
\end{lemma}
From this Lemma we also get the following tightness criterion.
\begin{lemma}
\label{Lem2BrzFer}
We are given parameters $\gamma>0$, $\delta>0$, $\alpha, q>1$ 
and a sequence $\{f_n\}_{n \in \mathbb{N}}$ of adapted processes in $C(\left[0,T\right];U')$. 
\newline
Assume that for any $\varepsilon>0$ there exist positive constants $R_i=R_i(\varepsilon)$ $(i=1,...,4)$ such that 
\begin{align*}
&\sup_n \mathbb{P}\left( \|f_n\|_{L^{\alpha}(0,T;L^q)}>R_1\right) \le \varepsilon
\\
&\sup_n \mathbb{P}\left( \|f_n\|_{C^{\gamma}(\left[0,T\right];W^{-1,2})}>R_2\right) \le \varepsilon
\\
&\sup_n \mathbb{P}\left( \|f_n\|_{L^2(0,T;W^{\delta,2})}>R_3\right) \le \varepsilon
\\
&\sup_n \mathbb{P}\left( \|f_n\|_{L^{\infty}(0,T;L^2)}>R_4\right) \le \varepsilon
\end{align*}
Let $\mu_n$ be the law of $f_n$ on $Z=L^{\alpha}_w(0,T;L^q) \cap C(\left[0,T\right];U') \cap L^2(0,T; L^2_{loc})\cap C(\left[0,T\right];L^2_w)$. Then the sequence $\{\mu_n\}_{n \in \mathbb{N}}$ is tight in $Z$.
\end{lemma}
\begin{remark}
\label{ASD...}
Lemma \ref{Lem2BrzFer} holds true also for the case of divergence free vector field spaces.
\end{remark}

We are now ready to prove the main Theorem. We point out that, differently from \cite{MikRoz2004}
dealing with $L^p(\mathbb{R}^d)$-valued solutions for $p>d$,   our result provides 
$ L^2(\mathbb{R}^2)$-valued solutions $\xi$ if $\xi_0 \in L^2$, $v_0\in \mathbb{L}^2$ and 
$L^q(\mathbb{R}^2)\cap L^2(\mathbb{R}^2)$-valued solutions $\xi$  if $\xi_0 \in L^2 \cap L^q$, $v_0\in \mathbb{L}^2 \cap \mathbb{L}^q$, for $q>2$. 

Formally, the results of Sections \ref{tight_sec_v} and \ref{tight_sec_xi}, Lemma \ref{Lem2BrzFer} 
and Remark \ref{ASD...} provide the tightness to pass to the limit. 
Proceeding similarly as in the proof of \cite[Theorem 3.6.]{BrzFer} we get the following result.
\begin{theorem}
\label{vor_bad}
\begin{enumerate}[label=\roman{*}), ref=(\roman{*})]
\item Let $q=4$ and assume conditions \textbf{(IG1)}, \textbf{(IG2)}, \textbf{(IG3)} and \textbf{(IG4)}. Let $\xi_0 \in L^2$ and $v_0 \in \mathbb{L}^2$. Then there exists a martingale solution $((\tilde{\Omega}, \tilde{\mathcal{F}}, \tilde{\mathbb{P}}), \tilde W,\tilde{\xi})$ (in the sense of Definition \ref{mar_sol}) to \eqref{vort}. In addition $\tilde{\xi} \in L^4(0,T;L^4)$ $\mathbb{P}$-a.s.. 
\item
If, in addition, we assume that condition \textbf{(IG2)} holds also for a $q>2$, and $\xi_0 \in L^2\cap L^q$, $v_0 \in \mathbb{L}^2\cap \mathbb{L}^q$, then also $\tilde{\xi} \in L^{\infty}(0,T;L^q)$ $\mathbb{P}$-a.s..
\end{enumerate}
\end{theorem}

\begin{proof}
Let us prove (i). One proceeds as in \cite{BrzFer}. 
We fix $0 < \gamma < \frac 12$ and $0 < \delta <1$ appearing in Proposition \ref{xi_n_prop} and define the spaces
\begin{equation*}
Z=L^{4}_w(0,T;L^4) \cap C(\left[0,T\right];U') \cap L^2(0,T; L^2_{loc})\cap C(\left[0,T\right];L^2_w),
\end{equation*}
\begin{equation*}
\mathbb{Z}=L^{4}_w(0,T;\mathbb{L}^4) \cap C(\left[0,T\right];\mathbb{U}') \cap L^2(0,T; \mathbb{L}^2_{loc})\cap C(\left[0,T\right];\mathbb{L}^2_w),
\end{equation*}
with the topology $\mathcal{T}$ and $\tau$ respectively,  given by the supremum of the corresponding topologies. According to Lemma \ref{Lem2BrzFer} (with $\alpha= 4$, $q=4$), Proposition \ref{xi_n_prop}(i) provides that the sequence of laws of the processes $\xi_n$ is tight in $Z$. Moreover, according to Lemma \ref{Lem2BrzFer} (with $\alpha=q=4$) and Remark \ref{ASD...}, Propositions \ref{v_n_prop} provide that the sequence of laws of the processes $v_n$ is tight in $\mathbb{Z}$.
So the pair $(\xi_n, v_n)$ is tight in $Z\times \mathbb{Z}$.
\newline
By the Jakubowski's  generalization of the Skorokhod Theorem in non metric spaces (see \cite{BrzMot}, \cite{Jak1} and \cite{Jak2}) there exist subsequences $\{\xi_{n_k}\}_{k=1}^{\infty}$ and $\{v_{n_k}\}_{k=1}^{\infty}$, a stochastic basis $(\tilde{\Omega}, \tilde{\mathcal{F}}, \tilde{\mathbb{P}})$, $Z$-valued Borel measurable variables $\tilde \xi$ and $\{\tilde {\xi}_k\}_{k=1}^{\infty}$, $\mathbb{Z}$-valued Borel measurable variables $\tilde v$ and $\{\tilde {v}_k\}_{k=1}^{\infty}$ such that
\begin{itemize}
\item the laws of $\xi_{n_k}$ and $\tilde{\xi}_k$ are the same and $\tilde{\xi}_k$ converges to $\tilde \xi$ $\tilde{\mathbb P}$-a.s. with the topology $\mathcal{T}$
\item the laws of $v_{n_k}$ and $\tilde{v}_k$ are the same and $\tilde{v}_k$ converges to $\tilde v$ 
with the topology $\tau$.
\end{itemize}
Since each $\tilde{\xi}_k$ has the same law as $\xi_{n_k}$, it is a martingale solution to \eqref{vort_ap}; therefore each process 
\begin{equation*}
\tilde{M}_k(t)= \tilde{\xi}_k(t)-\tilde{\xi}(0) +\int_0^t A\tilde{\xi}_k(s)\, {\rm d}s
+ \int_0^t \tilde{v}_k(s)\cdot \nabla \tilde{\xi}_k(s)\, {\rm d}s
\end{equation*}
is a martingale with quadratic variation
\begin{equation*}
\ll \tilde{M}_k\gg(t)=\int_0^t \tilde{G}_k(\tilde {v}_k(s)) \tilde{G}_k(\tilde {v}_k(s))^*\, {\rm d}s.
\end{equation*}
Proceeding as in \cite{BrzFer} we can prove that 
\begin{equation*}
\langle \tilde{M}_k(t)-\tilde{M}(t),\varphi\rangle \rightarrow 0 \qquad \tilde {\mathbb P}-a.s.
\end{equation*}
for any $\varphi \in H^{s,2}$, 
 with $s>2$, with compact support, and every $t \in \left[0,T\right]$, where
 \begin{equation*}
\tilde{M}(t)= \tilde{\xi}(t)-\tilde{\xi}(0) +\int_0^t A\tilde{\xi}(s)\, {\rm d}s
+ \int_0^t \tilde{v}(s)\cdot \nabla \tilde{\xi}(s)\, {\rm d}s.
\end{equation*}
In particular, the convergence of the non linear term 
\begin{equation*}
\langle\int_0^t \tilde{v}_k(s)\cdot \nabla \tilde{\xi}_k(s)\, {\rm d}s, \varphi \rangle
\rightarrow 
\langle \int_0^t \tilde{v}(s)\cdot \nabla \tilde{\xi}(s)\, {\rm d}s, \varphi\rangle 
\end{equation*}
is obtained with a slightly modification of the proof of \cite[Lemma B1]{BrzMot}, exploiting the convergence of $\tilde{v}_k$ in $C(\left[0,T\right];\mathbb{L}^2_{loc})$ and of $\tilde{\xi}_k$ in $C(\left[0,T\right];L^2_{loc})$. 
\newline
For the convergence of the quadratic variation process 
\begin{equation*}
\int_0^t \langle \tilde{G}_k(\tilde{v}_k(s))^*\varphi_1, \tilde{G}_k(\tilde{v}_k(s))^*\varphi_2\rangle_{\mathcal{H}}\, {\rm d}s
\rightarrow 
\int_0^t \langle \tilde{G}(\tilde{v}(s))^*\varphi_1, \tilde{G}(\tilde{v}(s))^*\varphi_2\rangle_{\mathcal{H}}\, {\rm d}s,
\end{equation*}
for any $\varphi_1, \varphi_2 \in H^{-g}$, we proceed exactly as in \cite[Theorem 3.6]{BrzFer}. 
\newline
Similar convergence results show that the limit is a martingale. Therefore, we conclude appealing to the usual martingale representation Theorem: there exists a cylindrical $\mathcal{H}$-Wiener process $\tilde w$ such that 
\begin{equation*}
\langle \tilde{M}(t), \varphi\rangle= \langle \varphi, \int_0^t \tilde{G}(\tilde{v}(s)\,{\rm d}\tilde{w}(s)\rangle= \int_0^t \langle G(\tilde{v}(s))^*\varphi, \, {\rm d}\tilde{w}(s)\rangle.
\end{equation*}
Therefore, $\tilde{\xi}$ is a martingale solution to \eqref{vort} and $\tilde{\xi} \in L^4(0,T;L^4)$ $\mathbb{P}$-a.s..

Statement (ii) follows from Proposition \ref{xi_n_prop}(ii) and Proposition \ref{v_n_prop2}.  We can infer the existence of a subsequence $\{\tilde \xi_k\}_k$ converging in $L^{\infty}_w(0,T;L^q)$. The limit process $\tilde \xi$ is the solution to \eqref{vort} and $\tilde \xi \in L^{\infty}(0,T;L^q)$.
\end{proof}

From the pathwise uniqueness for $v$, stated in Proposition \ref{MiR}, we infer pathwise uniqueness for $\xi$. In particular, pathwise uniqueness and existence of martingale solutions implies existence of strong a solution. 
\begin{corollary}
\label{ASD}
Assume that the same assumptions as Theorem \ref{vor_bad}(i) hold, moreover assume \textbf{(IG5)}. Then there exists a unique strong solution to \eqref{vort}.
 \end{corollary}
As a byproduct of Theorem \ref{vor_bad} we gain more regularity for the solution $v$ to equation \eqref{NSabs}.
\begin{corollary}
Under the same assumptions as Theorem \ref{vor_bad}(i) the solution process $v$ of \eqref{NSabs} has 
$\mathbb{P}$-a.s. paths in $C(\left[0,T\right];H^{1,2})$.
Moreover, under the same assumptions as Theorem \ref{vor_bad}(ii), $v$ has also $\mathbb{P}$-a.s. paths in $L^{\infty}(0,T;H^{1,q})$.
\end{corollary}

\appendix
\section{Study of the Navier-Stokes equations driven by a more regular noise}
\label{A}
In the present Appendix we are concerned with the Navier-Stokes equations \eqref{NS_0} driven by a more regular covariance operator $G$. The existence result we provide here is needed in the proof of Proposition \ref{prop3}.

On the covariance $G$ we make the following set of assumptions.
\begin{description}
\item[(G1)] the mapping $G:H^{1,2} \rightarrow L_{\text{HS}}(\mathcal{H};H^{1,2})$ is well defined and there exists $a_1>0$ such that 
\begin{equation*}
\|G(v)\|_{L_{\text{HS}}(\mathcal{H};H^{1,2})}\le a_1(1+\|v\|_{H^{1,2}}), \qquad \forall v \in H^{1,2}.
\end{equation*}
\item[(G2)] For all $z \in C^{\infty}_{sol}$ the real valued function $v \mapsto |G(v)^*z|_{\mathcal{H}}$ is continuous on $H^{1,2}$ endowed with the strong $L^2$ topology.
\end{description}
We give the following notion of solution to \eqref{NS_0}.
\begin{definition}
\label{mart_sol}
Let $v_0 \in H^{1,2}$.
A martingale solution to the Navier-Stokes problem \eqref{NS_0} is a triple consisting of a filtered probability space $(\Omega, \mathcal{F}, \{\mathcal{F}_t\}_{t \in \left[0,T\right]}, \mathbb{P})$, an $\{\mathcal{F}_t\}$-adapted cylindrical $\mathcal{H}$-Wiener process $W$ and an $\{\mathcal{F}_t\}$-adapted measurable $H^{1,2}$
-valued process $v$, such that
\begin{enumerate}[label=\roman{*}., ref=(\roman{*})]
\item for every $p \in \left[1, \infty \right)$, 
\begin{equation}
\label{blu}
v \in L^p(\Omega; L^{\infty}(0,T;H^{1,2} )) \cap L^2(\Omega;L^2(0,T;H^{2,2})) , \qquad \mathbb P-a.s.;
\end{equation}
\item for all $z\in C^{\infty}_{\text{sol}}$ and $t \in \left[0,T\right]$ one has $\mathbb{P}$-a.s.
\begin{equation}
\label{weak_NS}
\langle v(t),z \rangle =\langle v_0,z\rangle  
+ \int_0^t \langle \Delta v(s) , z\rangle \, {\rm d}s
+\int_0^t \langle (v(s) \cdot \nabla)v(s),  z\rangle\, {\rm d}s
+\langle \int_0^tG(v(s))\, {\rm d}W(s),z\rangle.
\end{equation}
\end{enumerate}
\end{definition}

In the definition of the martingale solution the incompressibility condition is contained in the requirement that $v$ belongs to $H^{1,2}$.
\begin{theorem}
\label{existence}
Assume that \textbf{(G1)} and \textbf{(G2)} hold. Then for any $v_0 \in H^{1,2}$ there exists a martingale solution to the problem \eqref{NS_0}.
\end{theorem}
Proof of Theorem \ref{existence} is a variation of proof \cite[Theorem 2.1]{BrzPsz}: there the authors consider the Euler equations 
on $\mathbb{R}^2$ perturbed by a multiplicative noise term satisfying the same assumptions we made. 
In order to prove the existence of  a martingale solution they consider a smoothed Faedo-Galerkin scheme of the Navier-Stokes equations. 
In particular a diffusion term $\nu \Delta v$, $\nu >0$, is added in order to use its smoothing effect and obtain the desired estimates. 
In the tightness argument, passing from the finite dimensional approximation to the infinite dimensional non approximated equation, 
they consider $\nu\rightarrow 0$ to recover the Euler equation in the limit.
The main difference in our result is that we maintain the regularizing effect of the Laplacian also in the limit equation. 
In this way we prove more regularity for the solution. We provide only a sketch of the proof.

\begin{proof}
\textbf{Smoothed Faedo-Galerkin approximations.}
As usual we project the first equation of \eqref{NS_0} onto the space of divergence free vectors fields to get rid of the pressure term.
We approximate the nonlinear term $B$ and the covariance operator $G$ in such a way they become Lipschitz in appropriate functional spaces. 
We consider the same approximations $B_n$ and $G_n$ as \cite[Section 5]{BrzPsz}. We recall them here for the sake of clarity. 
\newline
Let $\{e_k\}_k\subset H^{2,2}$ be an orthonormal basis of $H^{1,2}$.
 Let $P^{(n)}$ and $P_n$ be the orthonormal projection of $H^{1,2}$ into the spaces $Span\{e_1,...,e_n\}$ and $Span\{e_n\}=\mathbb{R}e_n$ respectively. Let $\hat P^{(n)}:H^{1,2} \rightarrow \mathbb{R}$ be defined by $\hat P^{(n)}(v)e_n=P_n(v)$, $v \in H^{1,2}$. 
 \newline
Let us start by recalling the approximation of $G$. Let $\rho \in C^{\infty}_0(\mathbb{R})$ be a non-negative function with the support in $\left[0,1 \right]$ and such that $\int_{\mathbb{R}} \rho(x)\, {\rm d}x=1$. Let $\pmb{1}_n=\pmb{1}_{\left[-n,n\right]}$. Recall that, for all $\psi \in \mathcal{H}$ and for all $v \in H^{1,2}$, we have $G(v)\psi \in H^{1,2}$. For such $\psi$ and $v$ we define
\begin{align*}
\left[G_n(v)\psi \right]= n^{-n}P^{(n)}\int_{\mathbb{R}^n}\left[G\left(\sum_{i=1}^nx_ie_i \right)\psi \right]&\pmb{1}_n\left(\left|\sum_{i=1}^n x_ie_i \right|_{H^{1,2}} \right) 
\\
&\times \rho\left(n (\hat{P}^{(1)}v-x_1)\right)\cdot \cdot \cdot \rho\left(n (\hat{P}^{(n)}v-x_n)\right)\, {\rm d}x_1\ldots {\rm d}x_n.
\end{align*}
$G_n(\cdot)$ is bounded and globally Lipschitz from $H^{1,2}$ into  $L_{\text{HS}}(\mathcal{H};H^{1,2})$ (with bounds possibly depending on $n$).
Let now consider the approximation of the nonlinear term $B$. Let $\varphi_n:H^{1,2} \rightarrow H^{1,2}$ be defined by
\begin{equation*}
\varphi_n(u):=
\begin{cases}
u,  &\quad \text{if} \  \|u\|_{H^{1,2}} \le n
\\
n\|u\|^{-1}_{H^{1,2}}u, &\quad \text{otherwise}
\end{cases}
\end{equation*}
Define $B_n(v,v):=B(\varphi_n(v),v)$. $B_n$ is a globally Lipschitz map from $H^{1,2}$ to $\mathbb{L}^2$. 

We consider a sequence of finite dimensional stochastic differential equations, the (smoothed) Faedo-Galerkin systems
\begin{equation}
\label{smoothed}
\begin{cases}
\displaystyle  {\rm d}v^n(t)+\left[P^{(n)}A v^n(t)+P^{(n)}B_n(v^n(t),v^n(t)) \right]{\rm d}t
       =G_n(v^n(t))\,{\rm d} W(t),&\quad t \in [0,T] 
\\ 
\nabla \cdot  v^n(t)=0, &\quad t \in [0,T] 
\\
\displaystyle v^n(0)=v^n_0 &
\end{cases}
\end{equation}
Since all the coefficients are Lipschitz, for every $n \in \mathbb{N}$ this SDE admits a unique solution $v^n$. 
The crucial point is to prove the desired estimates uniformly in $n \in \mathbb{N}$.
Proceeding as in \cite[Lemma 5.1]{BrzPsz} and using the smoothing effect of the Laplacian operator we obtain that, for any $p \in \left[1, \infty\right)$ there exists a finite constant $C_1$, independent of $n$, such that 
\begin{equation}
\label{1n}
\mathbb{E} \sup_{t \in \left[0,T\right]} \|v^n(t)\|^p_{H^{1,2}}+
\mathbb{E}\int_0^T \|\nabla v^n(t)\|^2_{H^{1,2}}\, {\rm d}t \le C_1,
\end{equation}
for any $n \in \mathbb N$.

\textbf{Tightness.} Each process $v^n$ is defined on a filtered probability space $(\Omega, \mathcal{F}, \mathcal{F}_t, \mathbb{P})$ and satisfies \eqref{smoothed} driven by a cylindrical $\mathcal{H}$-Wiener process $W$. Let us denote by $\mathcal{L}(v^n)$ the law of $v^n$ on the space of trajectories $C(\left[0,T\right];H^{1,2})$. We aim at proving that this sequence is tight on an appropriate functional space.
If we consider an unbounded domain, the embedding of the Sobolev space of functions with square integral gradient into the $L^2$ space, unlike in the bounded case, is not compact. Compactness is crucial in a tightness argument. As in \cite{BrzPsz} we introduce spaces with weights.
Let $\theta \in C^{\infty}(\mathbb{R}^2)$ be a strictly positive even function equal to e$^{-|x|}$ for $|x|\ge 1$, and let us denote by $L^2_{\theta}$ the weighted space $\left[L^2(\mathbb{R}^2;\theta(x)\, {\rm d}x)\right]^2$.

Let us set
\begin{equation*}
M^n(t):=\int_0^tG_n(v^n(s))\, {\rm d}W(s), \qquad t \in \left[0,T\right],
\end{equation*}
and let $\mathcal{L}(M^n)$ be the law of $M^n$ on $C(\left[0,T\right];H^{1,2})$. From \cite[Lemma 6.3]{BrzPsz} we get that the family $\{\mathcal{L}(M^n)\}_{n \in \mathbb{N}}$ is tight in $C(\left[0,T\right];L^2_{\theta})$.
A classical result is that $\{M^n(t)\}_n$ are square integrable continuous $L^2_{\theta}$-martingales with quadratic variation
\begin{equation*}
\ll M^n(t)\gg \ =\int_0^t \left[( j_{H^{1,2};L^2_{\theta}}G_n(v^n(s))) ( j_{H^{1,2};L^2_{\theta}}G_n(v^n(s)))^*\right]\, {\rm d} s,
\end{equation*}
where $j_{H^{1,2};L^2_{\theta}}$ denotes the imbedding of $H^{1,2}$ into $L^2_{\theta}$. From \cite[Corollary 6.1]{BrzPsz} we infer that the family $\mathcal{L}(\ll M^n(t))\gg)\}_n$ of the laws of $\{\ll M^n(t)\gg)\}_n$ is tight in $C(\left[0,T\right];L_1(L^2_{\theta},L^2_{\theta}))$, where by $L_1(L^2_{\theta},L^2_{\theta}))$ we denote the space of nuclear operators from $L^2_{\theta}$ into $L^2_{\theta}$. Moreover, from \cite[Lemma 6.4]{BrzPsz} it follows that the family $\{\mathcal{L}(v^n)\}_{n\in \mathbb{N}}$ is tight in $L^2(0,T;L^2_{\theta})$.

\textbf{Convergence.}
Let $\tilde {\mathcal{H}}$ be a Hilbert space such that $\mathcal{H} \hookrightarrow \tilde {\mathcal{H}}$ with a Hilbert-Schmidt imbedding. Then $W$ is a process with continuous trajectories on $\tilde{\mathcal{H}}$.
Set
\begin{equation*}
\mathcal{A}=L^2(0,T;L^2_{\theta}) \times C(\left[0,T\right];L^2_{\theta}) \times C(\left[0,T\right];L^1(L^2_{\theta};L^2_{\theta})) \times C(\left[0,T\right]; \tilde{\mathcal{H}}).
\end{equation*}
From what stated above it follows that the family of laws $\{\mathcal{L}(v^n, M^n, \ll M^n\gg, W)\}_n$ of $\{(v^n, M^n, \ll M^n\gg, W)\}_n$ on $\mathcal{A}$ is tight. Hence, by the Prokhorov theorem it is relatively weakly compact. So, there exists a subsequence $\{n_l\}_{l \in \mathbb{N}}$ such that $\{(v^{n_l}, M^{n_l}, \ll M^{n_l}\gg, W)\}_{n_l}$ converges weakly as $l \rightarrow \infty$.

By the Skorokhod imbedding theorem there exists a probability space $\mathcal{Y}=(\tilde{\Omega}, \tilde{\mathcal{F}}, \{\tilde{\mathcal{F}_t}\}_t, \tilde{\mathbb{P}})$, random elements in $\mathcal{A}$, $(v,M,m,V)$ and $\{v^l,M^l,m^l,V^l\}_{l\in \mathbb{N}}$, defined on $\tilde{\Omega}$, such that 
\begin{description}
\item[(S1)] the laws of $(v^{n_l},M^{n_l},\ll M^{n_l}\gg, W)$ and $(v^{l},M^{l},\ll M^{l}\gg, V^l)$ are the same, 
\item[(S2)] $(v^{l},M^{l},\ll M^{l}\gg, V^l)\rightarrow(v,M,\ll M\gg, V)$, $\tilde{\mathbb{P}}$-a.s. in $\mathcal{A}$. 
\end{description}
From (S1) it follows, in particular, that $v^l$ is the solution to the appropriate Navier-Stokes equations \eqref{smoothed} driven by $V^l$. 
Moreover, for any $p \in (1, \infty)$,
\begin{equation*}
\sup_{l \in \mathbb{N}}\tilde{\mathbb{E}} \left[ \sup_{t \in \left[0,T\right]}\left(\|v^l(t)\|^p_{H^{1,2}}\right)+\int_0^T\|\nabla v^l(t)\|^2_{W^{1,2}}\,{\rm d}t\right]< \infty.
\end{equation*}
\newline
To prove that the limit $(v,M,\ll M\gg, W)$ is the martingale $H^{1,2}$-valued solution to the Navier-Stokes problem \eqref{NS_0} one can proceed as in \cite[Theorem 2.1]{BrzPsz}. The only relevant different part concerns the passage to the limit for the diffusion term. In \cite{BrzPsz}, as $l \rightarrow \infty$ this latter term tends to zero and the Euler equation is recovered. We get instead the Navier-Stokes equations. 

In this way we construct a filtered probability space, an adapted cylindrical $\mathcal{H}$-Wiener process $W$ 
and an adapted measurable $H^{1,2}$-valued process $v$ satisfying (i)-(ii) of Definition \ref{mart_sol}.
\end{proof}

\end{document}